\documentclass[article]{siamart220329}
\usepackage{multirow}
\usepackage{lmodern}
\usepackage[T1]{fontenc}
\usepackage{listings}
\usepackage{amsmath}
\usepackage{amssymb}
\usepackage{url}
\usepackage{latexsym}
\usepackage{mathtools} 

\usepackage[square,comma,numbers,sort&compress]{natbib}

\usepackage{lineno}
\usepackage{algorithm} 
\usepackage{algpseudocode}

\newlength\myindent
\newcounter{todocounter}

\newcommand{\R}{\mathbb{R}}

\newcommand{\ignore}[1]{}
\def\true{_{{true}}}

\newcommand{\rev}[1]{\textcolor{black}{#1}}

\title{H-CMRH: An Inner Product Free Hybrid Krylov Method for Large-Scale Inverse Problems}

\author{
Ariana~N.~Brown\thanks{Department of Mathematics, 
Emory University.  
Email: abro299@emory.edu.} \and
Malena Sabat\'e Landman\thanks{Department of Mathematics, 
Emory University.  
Email: msabat3@emory.edu.} \and
James~G.~Nagy\thanks{Department of Mathematics,
Emory University.  
Email: jnagy@emory.edu. This work was partially supported by the U.S.~National Science Foundation, Grants DMS-2038118 and DMS-02208294.}}

\begin{document}

\maketitle

\begin{abstract}
This study investigates the iterative regularization properties of two Krylov methods for solving large-scale ill-posed problems: the changing minimal residual Hessenberg method (CMRH) and a 
\rev{new} 
hybrid variant called the hybrid changing minimal residual Hessenberg method (H-CMRH). Both methods share the advantages of avoiding inner products, making them efficient and highly parallelizable, and particularly suited for implementations that exploit 
\rev{low }and mixed precision arithmetic. Theoretical results and extensive numerical experiments 
suggest that H-CMRH exhibits comparable performance to the established hybrid GMRES method in terms of stabilizing semiconvergence, but H-CMRH 
does not require any inner products, and requires less work and storage per iteration.
\end{abstract}



\pagestyle{myheadings} \thispagestyle{plain} \markboth{
A.~N.~BROWN, M.~SABAT\'E LANDMAN, J.~G.~NAGY}{H-CMRH: Inner Product Free Hybrid Krylov Regularization}

\section{Introduction}

Linear large-scale discrete inverse problems can be written in the form
\begin{equation} \label{eq:1} 
Ax\true + e = b\,,
\end{equation} 
where $A \in \mathbb{R}^{n \times n}$ models the forward problem, $x\true \in \mathbb{R}^n$ is the unknown solution we want to approximate, $b \in \mathbb{R}^{n}$ is the vector of measurements, and $e \in \mathbb{R}^n$ represents errors in the measured data 
\rev{
(e.g., noise, measurement and discretization errors), and is typically
characterized as} independent and identically distributed (iid) zero mean normal random variables (so-called Gaussian white noise). Inverse problems of the form (\ref{eq:1}) arise in a variety of applications, including medical and geophysical imaging, 
and image deblurring \cite{Hansen2010, Vogel2002, Chung2011Inverse,Zhdanov2002}. Solving (\ref{eq:1}) is cumbersome due to its ill-posed nature, meaning that the singular values of the system matrix $A$ decay and cluster at zero without a gap to indicate numerical rank. This property, namely $A$ being moderately to severely ill-conditioned, causes the solution to be sensitive to perturbations in the measurement vector $b$; that is, small changes in $b$ can produce large changes in the estimate of the true solution $x\true$.

Regularization is typically used to alleviate this sensitivity \cite{Hansen2010} and enables the production of meaningful approximations of the true solution. For example, a very relevant approach when dealing with large-scale problems is to employ iterative regularization, 
which consists of applying an iterative solver to the least squares problem
\begin{equation} \label{eq:3}
    \min_{x \in \R^{n}} \| b - Ax \|^{2}\,,
\end{equation} 
where regularization is achieved through early termination of the iterations \cite{chung2024computational}. 
The stopping iteration acts \rev{as} a regularization parameter: 
if the iterative solver is terminated too soon, the solution is overly smooth, whereas the solution is highly oscillatory if the iterative solver is terminated too late. 

Alternatively, one can use a variational approach such as Tikhonov regularization, which consists of solving an optimization problem of the form
\begin{equation}  \label{eq:2} 
\min_{x \in \mathbb{R}^n} \| b - Ax \|^2 + \lambda^{2} \|Wx\|^2\,, 
\end{equation}
where $\lambda$ is the regularization parameter and $\|Wx\|^2$  \rev{is a regularization term}, where $W \in \mathbb{R}^{n \times n}$ \rev{is often chosen to be the identity matrix or a discrete approximation of a differentiation operator \cite{chung2024computational}.} 
Similarly to the role of the stopping iteration in iterative regularization, it is important to choose a good value for $\lambda$; if it is chosen too large, the regularized solution is overly smooth, and a choice of $\lambda$ that is too small produces a highly oscillatory solution. Moreover, 
multiple linear solves might be needed to refine the choice of $\lambda$, drastically increasing the computational cost. 
Thus, when dealing with large scale problems, it may be necessary to solve (\ref{eq:2}) with an iterative method. 

Hybrid regularization is a particular combination of variational and iterative regularization. This approach consists of iteratively projecting (\ref{eq:3}) onto a small subspace of increasing dimension and applying variational (e.g., Tikhonov) regularization to the small projected problem. In this framework, we only require matrix-vector products with $A$, which allows us to avoid explicitly constructing or storing $A$. It also creates a natural environment, namely the small projected problems, for estimating a good regularization parameter \cite{chung2024computational}.

\rev{Currently, all existing hybrid regularization algorithms require inner products. In some cases, this can be a computational burden, for example in distributed memory implementations with a large number of processors, the inner products (requiring global-communication) can be a limiting factor for efficiency, see, e.g. \cite{Saad1989Supercomputers, 1994templates}. Moreover, inner products can also affect the performance of algorithms in low and/or mixed precision arithmetic, where norms of large vectors can lead to under- or overflow, and where the notion of numerical orthogonality depends on the working floating point arithmetic, which can lead to early stopping of the traditional methods. A few inner product free iterative regularization methods exist but generally exhibit very slow convergence properties, or require information that might be difficult to accurately estimate for large-scale problems. The most popular family of algorithms of this kind consists of Chebyshev semi-iteration methods, see e.g. \cite{VARGA1961, Bjorck1996Numerical, 1994templates}, which require spectral knowledge about the system matrix $A$.  Alternatively, two simpler algorithms that do not require inner products are Landweber \cite[Chapter 6.1.1]{Hansen2006deblurring} and Richardson (first-order) \cite[Chapter 7.2.3]{Bjorck1996Numerical} methods. Note that the first-order Richardson method corresponds to applying Landweber to the normal equations. Moreover, this is equivalent to gradient descent with a fixed step length that depends on spectral bounds for $A$, requires positive definiteness of $A^T A$ to converge, and are generally slow to converge.
}

\rev{
In this context, the changing minimal residual Hessenberg method (CMRH) \cite{sadok1999new}  appears as a potential inner product free alternative. As a Krylov subspace method, it has much faster convergence properties than Chebyshev, Landweber and Richardson iterations. However, its suitability for ill-posed problems, in particular if it could be used as an iterative regularization method, has not been previously studied. In this paper, we analyze the CMRH iterative regularization properties for the first time and, given the discovered suitability of this method in this context, we also introduce a new hybrid projection algorithm known as the hybrid changing minimal residual Hessenberg method (H-CMRH).} 

\rev{In summary, this paper presents the following new contributions: 
\begin{itemize}
    \item 
    Establishing CMRH as an iterative regularization method and therefore opening its potential to a wide range of new applications. This in-depth study offers a combination of theoretical justifications based on bounds for the residual norm as well as a new understanding of the underlying properties of the projected problems. 
    We establish that CMRH is a regularization method that effectively acts as a filtering method, and this is supported by relevant empirical evidence. 
    Proposing a new hybrid version of the CMRH method, which is currently the only existing inner product free hybrid method. We emphasize that the construction of this algorithm is only meaningful after the regularization properties of CMRH have been established.
    \item Demonstrating the performance of CMRH in an array of ill-posed problems, including comparisons with other inner product free methods and GMRES. We also include a low-precision arithmetic case study, where we highlight certain pitfalls of GMRES that are avoided by using CMRH.
\end{itemize}}

An outline for this paper is as follows. In Section \ref{sec:CMRH}, we review the CMRH method and its relationship with the generalized minimal residual algorithm (GMRES). In Section \ref{sec:regularization_properties}, we describe the regularization properties of CMRH and provide illustrative numerical examples. In Section \ref{sec:H-CMRH} we introduce the H-CMRH method and discuss its properties. In Section \ref{sec:numerics} we illustrate the effectiveness of H-CMRH using different numerical examples of image deblurring problems and provide some closing remarks in Section \ref{sec:conclusions}. Throughout the paper, we assume that $\| \cdot \|$ is the Euclidean norm.

\section{The Changing Minimal Residual Hessenberg Method}\label{sec:CMRH}
Both CMRH in \cite{sadok1999new} and GMRES \cite{saad2003iterative} iteratively approximate the solution of (\ref{eq:3}) on \rev{the following} Krylov subspace of increasing dimension: $$ {\cal{K}}_{k} = span \{r_0, Ar_0,A^2r_0, .....,A^{k-1}r_0 \} $$ 
where $r_0 = b-Ax_0$ and $x_0 $ is an initial guess for the solution.

Within each of these two iterative solvers, different basis vectors of ${\cal{K}}_k$ are formed. On one hand, CMRH utilizes the Hessenberg process, which produces linearly independent basis vectors. On the other hand, for GMRES, the Arnoldi algorithm forms orthonormal basis vectors of ${\cal{K}}_k$. Despite the lack of orthonormal vectors, CMRH is still advantageous. For dense problems, the matrix-vector product with $A$ can be calculated without accessing 
\rev{the first $k-1$ columns of $A$
(this is because, as will be seen in Algorithm~\ref{alg:arbitrarySample1}, the basis vectors contain zeros in the first $k-1$ entries), thus giving possible savings in the cost of the matrix-vector product \cite{sadok2011new}.} 
More generally, CMRH does not require inner products, which is important for computer efficiency, parallel computing and possibly reducing errors caused by inner products in low precision.

\subsection{The Hessenberg Process and CMRH}
Consider the following Krylov matrix: 
\begin{equation} \label{eq:4} 
V_k = [r_0, Ar_0, A^2r_0, .........,A^{k-1}r_0] \in \mathbb{R}^{n \times k}, 
\end{equation} 
where $V_{k+1} =  [r_0, \,AV_k]$ and clearly range$(V_k) = {\cal{K}}_k$. Note that $V_k$ becomes very ill-conditioned even for small numbers of $k$, so this matrix is not constructed explicitly, and is only considered theoretically to motivate the Hessenberg process, and later on its relationship with the Arnoldi process.

The Hessenberg process can be derived by considering the LU factorization of \begin{equation} \label{eq:LU} 
V_{k} = L_{k}U_{k},
\end{equation}
where $L_k \in \mathbb{R}^{n \times k}$ is a unit lower triangular matrix and $U_{k}$ is an upper triangular matrix. However, the algorithm does not explicitly compute this factorization but instead recursively computes the columns of $L_k$. Following \cite{sadok2011new} we can write the following relation: 
\begin{equation} \label{eq:5} 
V_{k+1} \begin{bmatrix} 0_{1 \times k} \\ I_{k}\end{bmatrix} = L_{k+1}U_{k+1} \begin{bmatrix} 0_{1 \times k} \\ I_{k} \end{bmatrix} = AV_{k} = AL_{k}U_{k},
\end{equation} 
where $0_{1 \times k}$ is a row vector of zeros of dimensions $1$$\times$$k$.  One can then define an upper Hessenberg matrix of the form: 
\begin{equation} \label{eq:6} 
H_{{k+1},k }= U_{k+1} \begin{bmatrix} 0_{1 \times k} \\ I_{k} \end{bmatrix}U_{k}^{-1} \in \mathbb{R}^{(k+1) \times k},
\end{equation} 
where $k < n$. 
Note that the columns of $L_{k}$ are linearly independent vectors and form a basis of the Krylov subspace ${\cal{K}}_k$. Moreover, combining (\ref{eq:5}) and (\ref{eq:6}) leads to the following Hessenberg relation: 
\begin{equation} \label{eq:7} 
AL_{k} = L_{k+1} H_{{k+1},k}. 
\end{equation}

An implementation of the Hessenberg process \cite{sadok2008new} can be found in Algorithm \ref{alg:arbitrarySample1}. The algorithm requires providing a nonzero vector $r_0$; we discuss below the specific choice for this vector in the CMRH method. 
\rev{Note that in line 3 of Algorithm~\ref{alg:arbitrarySample1}, the vector $l_k$ is a column of a lower triangular matrix, so its first $k-1$ entries are zero. Thus the matrix-vector product with $A$ does not need to access the first $k-1$ columns of $A$, which can give a possible savings in cost of the matrix-vector product. 
}

\begin{algorithm}[H]
\caption{Hessenberg Process} \label{alg:arbitrarySample1}
\begin{algorithmic}[1]
\State $\beta = e_{1}^{T}r_0$ (first entry of given, nonzero vector $r_0$); $l_1 = r_0 / \beta$
\For{$k = 1,....,m$}
\State $u = Al_k$
\For{$j = 1,....,k$}
\State $h_{j,k} = u(j)$; $u = u - h_{j,k}l_j$
\EndFor
\State $h_{k+1,k} = u(k+1)$; $l_{k+1} = u/h_{k+1,k}$
\EndFor
\end{algorithmic} 
\end{algorithm} 

Examining this implementation, we find that if $\beta = 0$ or $h_{k+1,k} = 0$, then the Hessenberg process will breakdown. To circumvent this breakdown, a pivoting approach is introduced in \cite{sadok2008new}. The Hessenberg process with pivoting is described in Algorithm \ref{alg:arbitrarySample2}.

\begin{algorithm}[H]
\caption{Hessenberg Process with Pivoting} \label{alg:arbitrarySample2}
\begin{algorithmic}[1]
\State $p = [1, 2, ...., n]^{T}$, and let $r_0$ be a given, nonzero vector
\State Determine $i_0$ such that $|r_0(i_0)| = \|r_0\|_{\infty}$
\State $\beta = r_0(i_0)$ ; $ l_1 = r_0/ \beta$; $p_1 \Leftrightarrow p_{i_0}$
\For{$k = 1,....,m$}
\State $u = A l_k$
\For{$j = 1,....,k$}
\State $h_{j,k} = u(p(j))$ ; $u = u - h_{j,k} l_j$
\EndFor
\If{$k < n$ and $u \neq 0$}
\State Determine $i_0 \in \{ k+1, ......., n\}$ such that $|u(p(i_0))| = \| u(p(k+1 : n)) \|_{\infty}$
\State $h_{k+1,k} = u(p(i_0))$; $l_{k+1} = u/ h_{k+1,k}$ ; $p_{k+1} \Leftrightarrow p_{i_0}$
\Else
\State $h_{k+1,k} = 0$; Stop.
\EndIf
\EndFor
\end{algorithmic} 
\end{algorithm} 

At each iteration $k$, CMRH finds an approximate solution of (\ref{eq:3})  by minimizing the following least squares problem: 
\begin{equation} \label{eq:8} 
\min_{x \in {\cal{K}}_{k}} \| L^{\dagger}_{k+1} (b - Ax) \| , 
\end{equation} 
where $L^{\dagger}_{k+1}$ is the pseudoinverse of the lower triangular matrix $L_{k+1}$ (which is constructed with the Hessenberg process). We observe that CMRH is a projection method since, at each iteration, it considers a solution on the Krylov subspace ${\cal{K}}_k$. However, due to $L_{k+1}^{\dagger}$ \rev{not having orthogonal columns}, 
the subproblems of the form \eqref{eq:8} solved at each iteration are oblique projections of the original problem (\ref{eq:3}), and therefore CMRH is not exactly a minimal residual method. For that reason, CMRH holds some parallels with the quasi-minimal residual method (QMR), in that the imposed optimality conditions involve minimizing a semi-norm \cite{sadok2016algorithms}. Considering $x=x_0+L_ky$, and $r_0 = b - A x_0$, the norm in (\ref{eq:8}) can be rewritten as:
\begin{eqnarray*}
  \| L^{\dagger}_{k+1} (b - A(x_0 + L_k y)) \| 
 &=& \|L^{\dagger}_{k+1} (r_0 - A L_k y) \| \\
  &=&  \| L^{\dagger}_{k+1} (r_0 - L_{k+1} H_{k+1,k}y) \| \\
   &=&  \| \beta e_{1} - H_{k+1,k}y \|, 
 \end{eqnarray*}
where we have used the Hessenberg relation \eqref{eq:6}, and where $\beta$ is the first entry in $r_0$ (if using Algorithm~\ref{alg:arbitrarySample1}) or the largest entry in $r_0$ (if using Algorithm~\ref{alg:arbitrarySample2}). Thus, at each iteration, we solve a small subproblem on a space of increasing dimension
\begin{equation*}  
y_k =  \arg \min_{y \in \mathbb{R}^{k}} \| \beta e_1 - H_{k+1\rev{,k}}y \| 
\end{equation*} 
and then project back using $x_k=x_0+L_ky\rev{_k}$.

\subsection{The Relationship Between CMRH and GMRES}

\rev{In this section, we recall the theoretical relationship between the basis vectors produced by GMRES and CMRH, as well as bounds on the difference between their relative residual norms. We include detailed derivations for completion, as well as to guide the proofs for the hybrid version of the algorithm presented in Section \ref{sec:H-CMRH}. However, this is not sufficient to explain the regularization properties of CMRH, as it is explained in Section \ref{sec:regularization_properties}}.

Similar to CMRH, one can derive the GMRES method by considering the QR factorization of the Krylov matrix defined in (\ref{eq:4}): 
\begin{equation}\label{eq:QR}
    V_k = Q_{k} \tilde{R}_k,
\end{equation}
where $Q_k \in \mathbb{R}^{n \times k}$ has orthonormal columns and $\tilde{R}_k$ is an upper triangular matrix. This algorithm does not explicitly compute the QR factorization, but instead recursively computes the columns of $Q_k$. From \cite{sadok2011new} we consider the following relationship: \begin{equation} \label{eq:10} V_{k+1} \begin{bmatrix}
    0_{1 \times k} \\ I_k
\end{bmatrix}  = Q_{k+1} \tilde{R}_{k+1} \begin{bmatrix}
    0_{1 \times k} \\ I_{k}
\end{bmatrix} = AV_{k} = AQ_{k}\tilde{R}_k, \end{equation}
where $0_{1 \times k}$ is a row vector of zeros of dimensions $1$$\times$$k$ and which produces an upper Hessenberg matrix of the form: 
\begin{equation} \label{eq:11} 
H_{k+1,k}^{A} = \tilde{R}_{k+1} \begin{bmatrix}
    0_{1 \times k} \\ I_{k}
\end{bmatrix} \tilde{R}_{k}^{-1} \in \mathbb{R}^{(k+1) \times k}, 
\end{equation} 
where $k<n$. Note that the columns of $Q_k$ form an orthonormal basis of ${\cal{K}}_k$. Moreover, combining (\ref{eq:10}) and (\ref{eq:11}) creates the Arnoldi relation:
\begin{equation} \label{eq:Arn} 
A_k Q_k = Q_{k+1}H_{k+1,k}^{A}. 
\end{equation}

The derivation of (\ref{eq:Arn}) and (\ref{eq:7}) naturally provides a mapping between the Arnoldi and  Hessenberg bases. Consider the LU and QR factorizations of the Krylov matrix $V_k$ in \eqref{eq:LU} and \eqref{eq:QR}, respectively, then we can define
\begin{equation} \label{eq:13} 
V_k = L_k U_k = Q_k \tilde{R}_k.
\end{equation}
Now, let $R_k = \tilde{R}_k U_{k}^{-1}$: this is a\rev{n} $k \times k$ upper triangular matrix. Using (\ref{eq:13}), the lower triangular matrix $L_k$ can be written as a composition of an orthogonal matrix and an upper right triangular matrix, 
$$L_k = Q_k \tilde{R}_k U_k^{-1}.$$
This corresponds to a QR factorization of $L_k$,
\begin{equation} \label{eq:14} 
L_k = Q_k R_k. 
\end{equation}

In addition to (\ref{eq:14}), we can rewrite the Arnoldi relation in (\ref{eq:Arn}) as:
\begin{equation} \label{eq:15} 
AL_k R^{-1}_k = L_{k+1}R^{-1}_{k+1} H_{k+1,k}^{A}. 
\end{equation}
Comparing the Hessenberg \eqref{eq:7} and Arnoldi \eqref{eq:15} relations, and following \cite{sadok2011new}, we provide the following proposition :
\begin{proposition}\label{prop:1}
    Let $H_{k+1,k}$ and $H^{A}_{k+1,k}$ be the Hessenberg matrices associated to the Hessenberg and Arnoldi processes, respectively, at iteration $k$, then $$H_{k+1,k} = R^{-1}_{k+1} H_{k+1,k}^{A} R_{k},$$ or, equivalently, $H^{A}_{k+1,k} = R_{k+1}H_{k+1,k}R_{k}^{-1}.$
\end{proposition}

Proposition~\ref{prop:1} is used to establish residual bounds for CMRH in the following sense: the residual norm associated to the approximated solution provided by CMRH at each iteration is close to the residual norm associated to the solution provided by GMRES if the condition number of $R_{k+1}$ does not grow too quickly. This can be observed in the following theorem, originally proved in \cite{sadok2011new}. Here, we restate it and provide additional details of the proof in order to set the stage for an analogous reasoning involving the hybrid CMRH algorithm (H-CMRH) in Section \ref{sec:H-CMRH}.

\begin{theorem}
Let $r_k^G$ and $r_k^C$ be the GMRES and CMRH residuals at the kth iteration beginning with the same initial residual $r_0$, respectively. Then 
\begin{equation}\label{eqthm} \|r_k^G\| \leq \|r_k^C\| \leq \kappa (R_{k+1}) \|r_k^G \| \end{equation} where $\kappa (R_{k+1}) = \|R_{k+1}\| \|R_{k+1}^{-1}\|$ is the condition number of $R_{k+1}$.
\end{theorem}

\begin{proof}

First, we prove the left inequality in (\ref{eqthm}). Consider the residual as a function of the solution: $$ r(x) = b - Ax.$$ Then, the kth residual norm associated to the approximated solution produced by GMRES is:
$$ \|r_k^G\| = \|b - Ax_k^G\| = \min_{x \in {\cal{K}}_k} \| r(x) \|.  $$ 
Since $x_k^G$ and $x_k^C$ are in the Krylov subspace ${\cal{K}}_{k}$, then by definition:
$$\min_{x \in {\cal{K}}_k} \|r(x) \| \leq \|r_k^C \| =\| r(x_k^C)\|. $$
Hence, $\|r_k^G\| \leq \|r_k^C \|$.

Now we prove the right inequality in (\ref{eqthm}). Since $r_k^C$ and $r_k^G$ are in the ${\cal{K}}_{k+1}$ subspace, we can write $r_k^C$ and $r_k^G$ as a linear combination of any basis of ${\cal{K}}_{k+1}$. Using the Hessenberg relation, the decomposition of the Krylov matrix is:  $$V_{k+1} = L_{k+1} U_{k+1}.$$ This implies range$(L_{k+1})$ = range$ (V_{k+1}) = {\cal{K}}_{k+1}$. 
Therefore, and using \eqref{eq:14}, there exist $u_k^C$ and  $w_k^C$ in $\mathbb{R}^{k+1}$ such that $$r_k^C= L_{k+1}u_k^C = Q_{k+1} R_{k+1} u_k^C = Q_{k+1}w_k^C$$ with $R_{k+1}u_k^C = w_k^C$. Analogously, there exist  $u_k^G$ and $w_k^G$ in $\mathbb{R}^{k+1}$ such that
\begin{equation}
\label{eq:pie} 
r_k^G = L_{k+1}u_k^G = Q_{k+1}R_{k+1}u_k^G = Q_{k+1}w_k^G \end{equation} 
with $R_{k+1} u_k^G = w_k^G$.

Consider the optimality conditions of CMRH. As stated above, $r_k^C = L_{k+1} u_k^C$. This implies that $L^{\dagger}_{k+1} r_k^C = u_k^C$. 
Hence, $\|L^{\dagger}_{k+1} r_k^C\| = \|u_k^C\|$ so 
\begin{equation}
\label{eq:pi} 
\|u_k^C\| = \min_{x \in {\cal{K}}_k} \|L^{\dagger}_{k+1} (b-Ax)\| =  \min_{x \in {\cal{K}}_k} \| L_{k+1}^{\dagger} r(x) \| 
\end{equation}
Using (\ref{eq:pi}) and the fact that $x_k^G$ is in ${\cal{K}}_k$ then  $\|u_k^C\| \leq \|u_k^G\|$. Thus $$ \|u_k^C\| \leq \|u_k^G\| = \|R_{k+1}^{-1}w_k^G\| \leq \|R_{k+1}^{-1}\| \|w_k^G\| = \|R^{-1}_{k+1}\| \|r_k^G\|,$$
where the equalities in the above relation come from (\ref{eq:pie}). On the other hand, and also using (\ref{eq:pie}): $$\|r_k^C\| = \|L_{k+1}u_k^C\| \leq \|L_{k+1}\| \|u_k^C\|. $$
Putting the above inequalities together gives the following relation:
\begin{eqnarray*}
  \| r_k^C \| 
 &=& \|L_{k+1} u_k^C \| \\
  &\leq&  \| L_{k+1} \|  \|u_k^C\| \\
   &\leq&  \| L_{k+1} \| \| R^{-1}_{k+1} \| \|r_k^G \|. 
 \end{eqnarray*}
Recall that $L_{k+1}$ has a QR decomposition (\ref{eq:14}) of the form $L_{k+1} = Q_{k+1}R_{k+1}$, where $Q_{k+1}$ is an orthogonal matrix. Therefore, $\|L_{k+1}\| = \|Q_{k+1}R_{k+1}\| = \|R_{k+1}\|$. This results in the following:  $$\|r_k^C\| \leq \|L_{k+1}\| \|R^{-1}_{k+1}\| \|r_k^G\| = \|R_{k+1}\|\|R^{-1}_{k+1}\| \|r_k^G\| = \kappa (R_{k+1})\|r_k^G\|. $$
Thus, we conclude that $\|r_k^G\| \leq \|r_k^C\| \leq \kappa (R_{k+1}) \|r_k^G\|$.
\end{proof}

\rev{
If we compare the computational cost between GMRES and CMRH, each algorithm requires one matrix-vector product per iteration, which is typically the most expensive computation in an iterative method. Additional costs related to vector operations are also similar, except that GMRES requires $k+1$ inner products at the $k$-th iteration, while no inner products need to be computed in CMRH. Note that, in the cases where the number of required iterations is large and matrix-vector products can be done very efficiently, the growing number of inner products required by GMRES can be a computation bottleneck. In these cases, CMRH presents a computational advantage with respect to GMRES.
}

\section{Regularizing Properties of  CMRH}\label{sec:regularization_properties}

\rev{
This section is devoted to providing a 
theoretical understanding 
that underlies the regularizing mechanisms of CMRH. These theoretical results are supported by 
(rather technical) numerical examples, which are designed to highlight the behaviour of CMRH in this context rather than its performance. 
Note that 
a variety of additional numerical examples on the performance of CMRH for ill-posed problems can be found in Section \ref{sec:numerics}.} 

As an iterative solver, CMRH produces a sequence of approximate solutions. In \rev{this paper we observe that, in} the early iterations, these approximations begin to converge toward the true solution, $x\true$; however, if the iterations are not stopped early, the approximate solutions eventually deviate from $x\true$ as they continue to converge to the solution $x = A^{\dagger} b = A^{\dagger}(Ax\true + e)$ of the least squares problem \eqref{eq:3}, which is a poor approximation of $x\true$ due to the ill-conditioning of $A$ and the noise, $e$. This is known as semiconvergence \cite{Hansen2010}. To overcome semiconvergence, we must stop CMRH at the right time. 

An analysis of the SVD of the projected matrix $H_{{k+1},k}$ from (\ref{eq:7}) provides a basis for understanding the regularizing properties of CMRH. Moreover, our hypothesis is that the approximated solutions produced by CMRH mirror regularized (filtered) solutions, so the spectral filtering properties of CMRH are also studied empirically in this section. 

\subsection{Singular Value Decomposition Analysis}
It has been observed that the singular values of the projected matrix $H^A_{k+1,k}$ in (\ref{eq:Arn}) for GMRES (and the corresponding projected matrix for QMR \cite{Gazzola2015OnKP}) tend to approximate the large singular components of $A$. This helps to explain the regularizing properties of these Krylov methods and the initial fast decay of the relative residual norms of QMR and GMRES \cite{Gazzola2015OnKP}. 

In this section, we compare the largest singular values of \rev{the} matrix $A$ with the singular values of the upper Hessenberg matrix $H_{{k+1},k}$ defined in (\ref{eq:7}) \rev{and} obtained at each iteration of CMRH. To analyze how well the singular values from $H_{{k+1},k}$ approximate those of $A$, in Figures \ref{fig:RegToolsSVD} and \ref{fig:RegToolsSV2} we plot the largest singular values of the full matrix $A$ using horizontal lines. The dots represent the singular values of the projected matrices at each iteration $k$ (see top plots of Figures \ref{fig:RegToolsSVD} and \ref{fig:RegToolsSV2}). As another way of visualizing this, the bottom plots display the singular values of the upper Hessenberg matrix at different iterations of CMRH against the singular values of $A$. Each figure represents a different $1$D inverse problem; the first three (Deriv2, Heat, and Shaw) are from the Regularization Tools Package~\cite{reg}, and the fourth example (Spectra) is a $1$D signal restoration problem, with a matrix modeling a Gaussian blur. Specifically, the entries of $A$ are given by
\begin{equation}
\label{eq:SpectraMatrix}
a_{ij} = \frac{1}{\varsigma \sqrt{2\pi}} \exp\left(-\frac{(i-j)^2}{2{\varsigma^2}}\right)\,,
\end{equation}
with $\varsigma = 2$, and $x\true$ is a simulated x-ray spectrum \cite{Trussell1983Convergence}.
\begin{figure}[htbp]
\begin{tabular}{cc}
    {\scriptsize Deriv2} &  {\scriptsize Heat} \\ 
    \includegraphics[width=6cm]{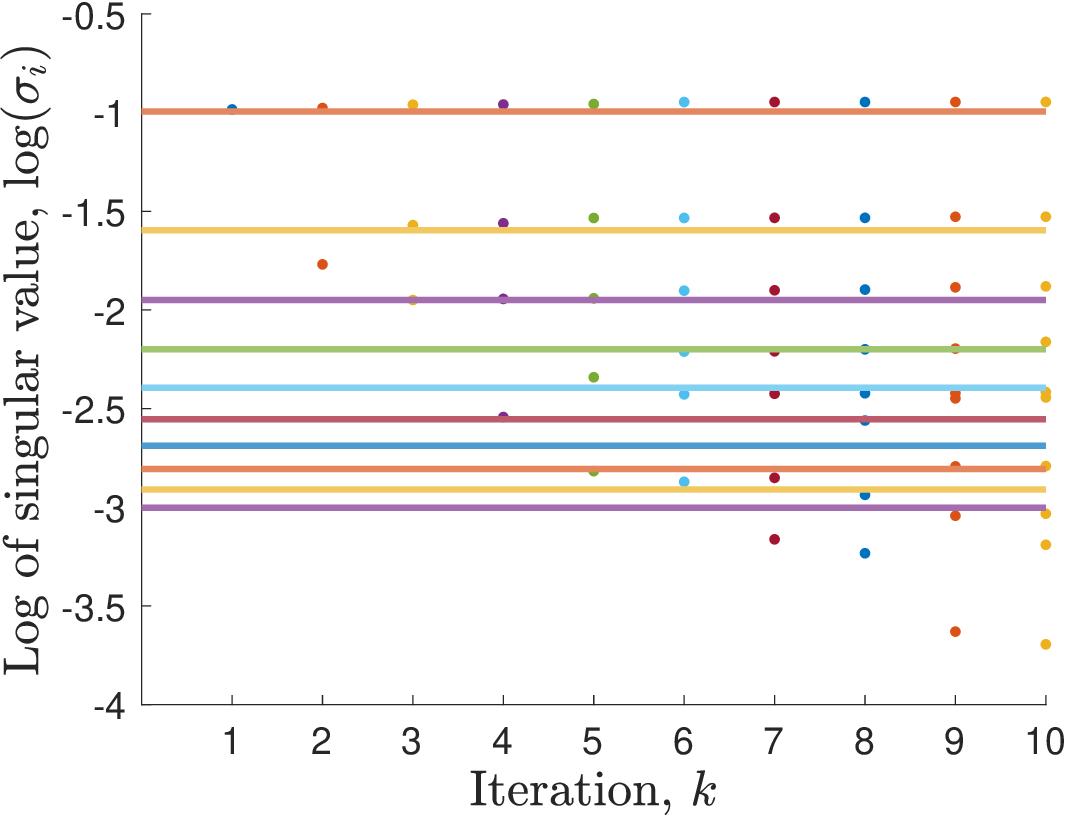} &  \includegraphics[width=6cm]{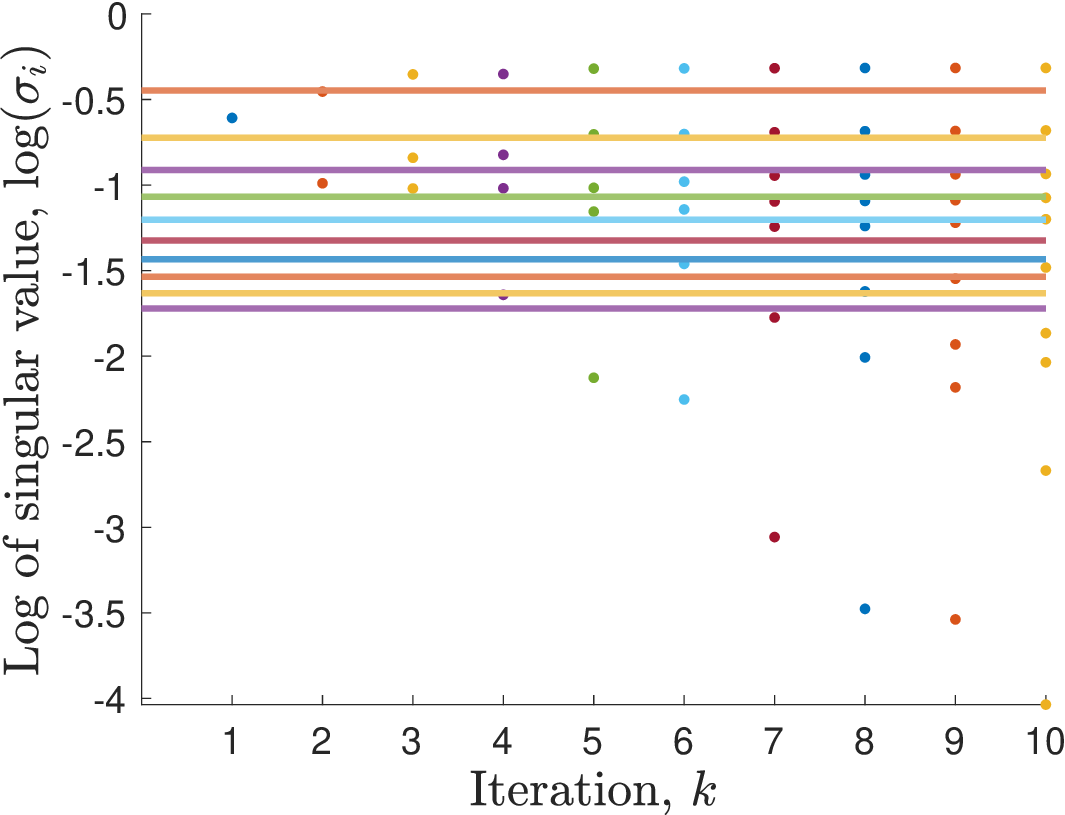} \\
     
\end{tabular} \\
\begin{tabular}{cc}
    \includegraphics[width=6cm]{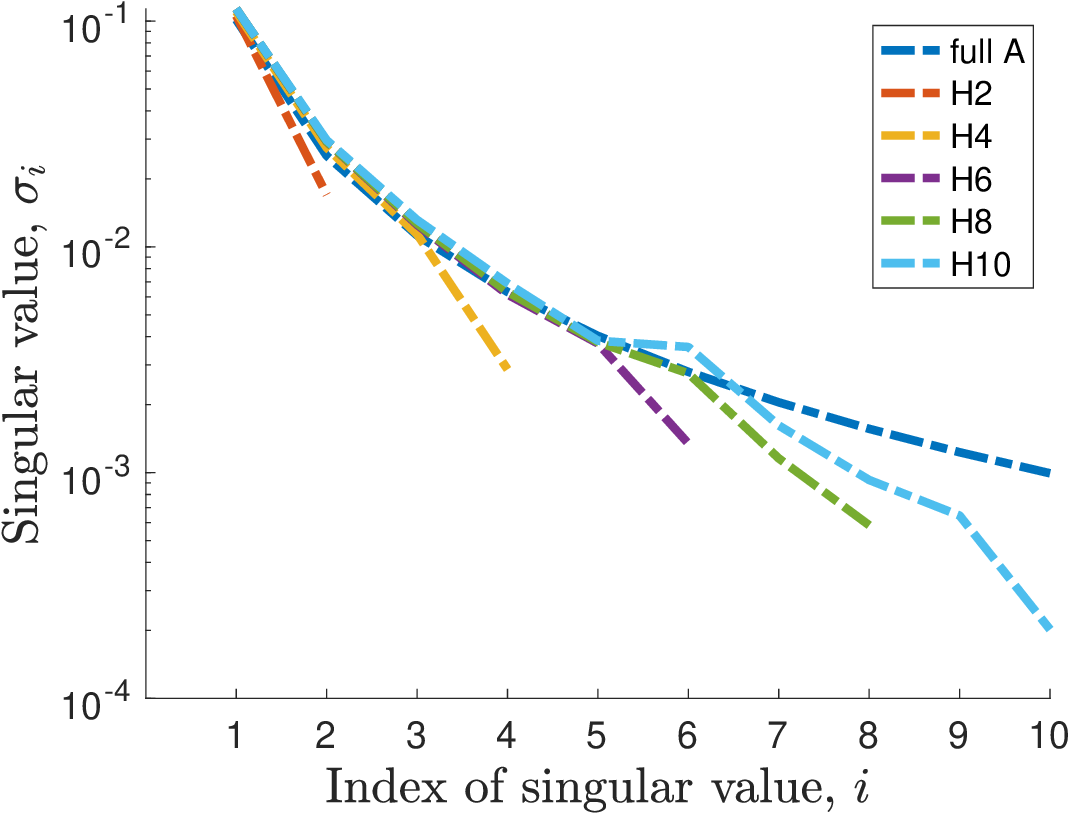} &  \includegraphics[width=6cm]{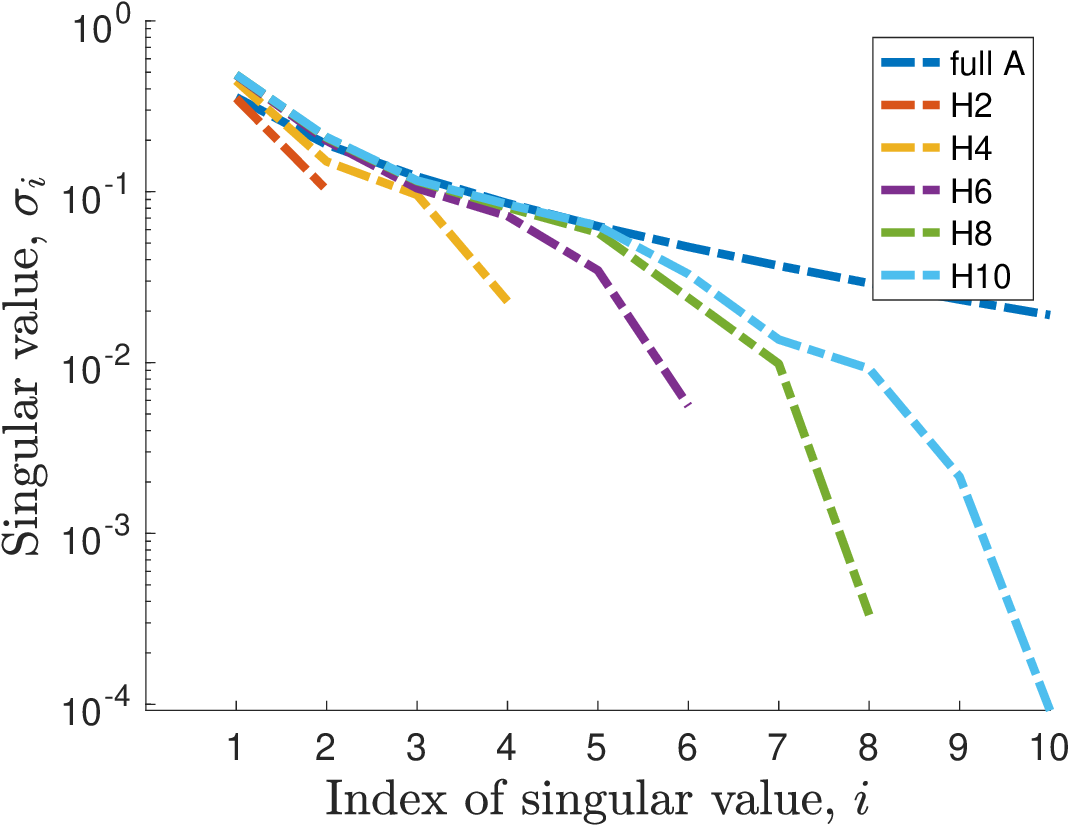} \\
\end{tabular}
\vspace{0.1cm}
\caption{Singular values for Deriv2-example 2 and Heat.}
\label{fig:RegToolsSVD}
\end{figure}

\begin{figure}[t]
\begin{tabular}{cc}
  {\scriptsize Shaw} & {\scriptsize Spectra} \\
     \includegraphics[width=6cm]{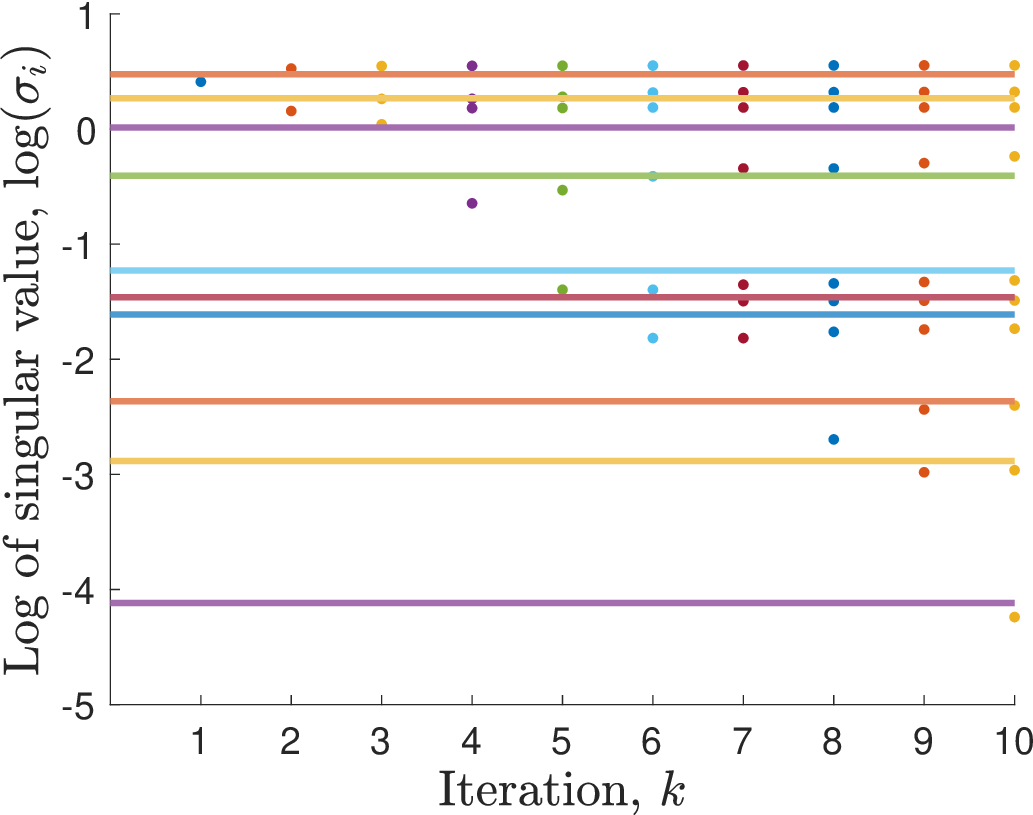} &
      \includegraphics[width=6cm]{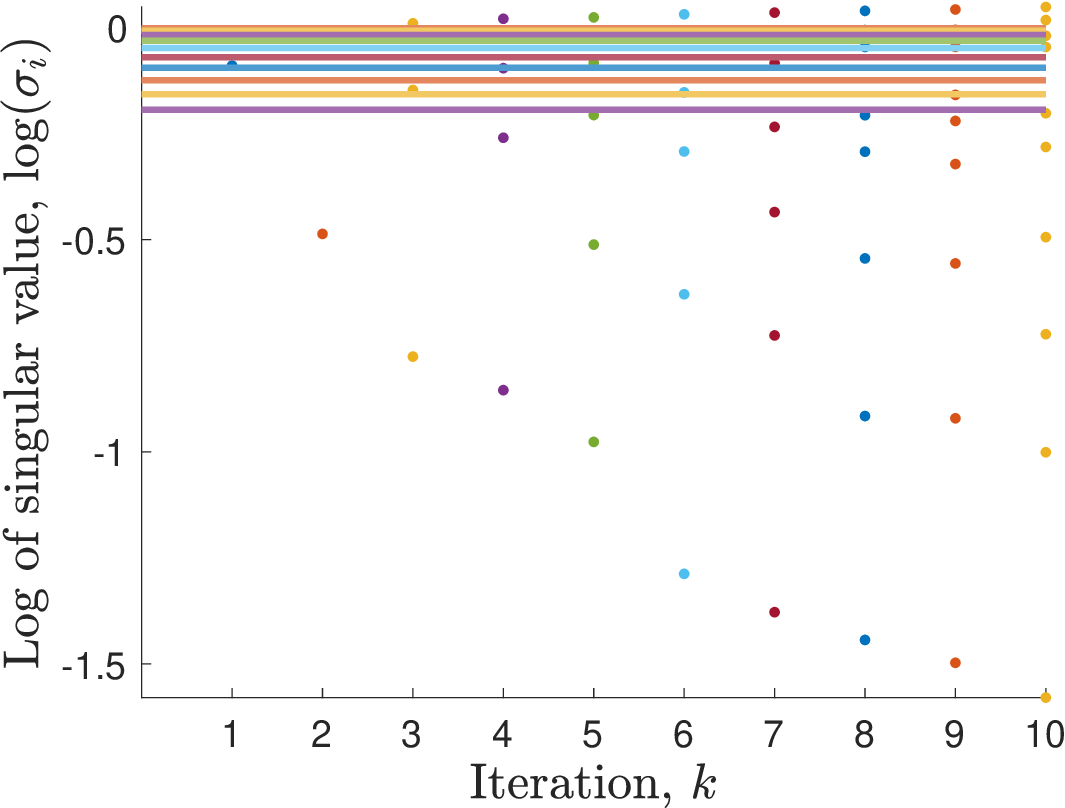}\\
\end{tabular} \\
\begin{tabular}{cc}
     \includegraphics[width=6cm]{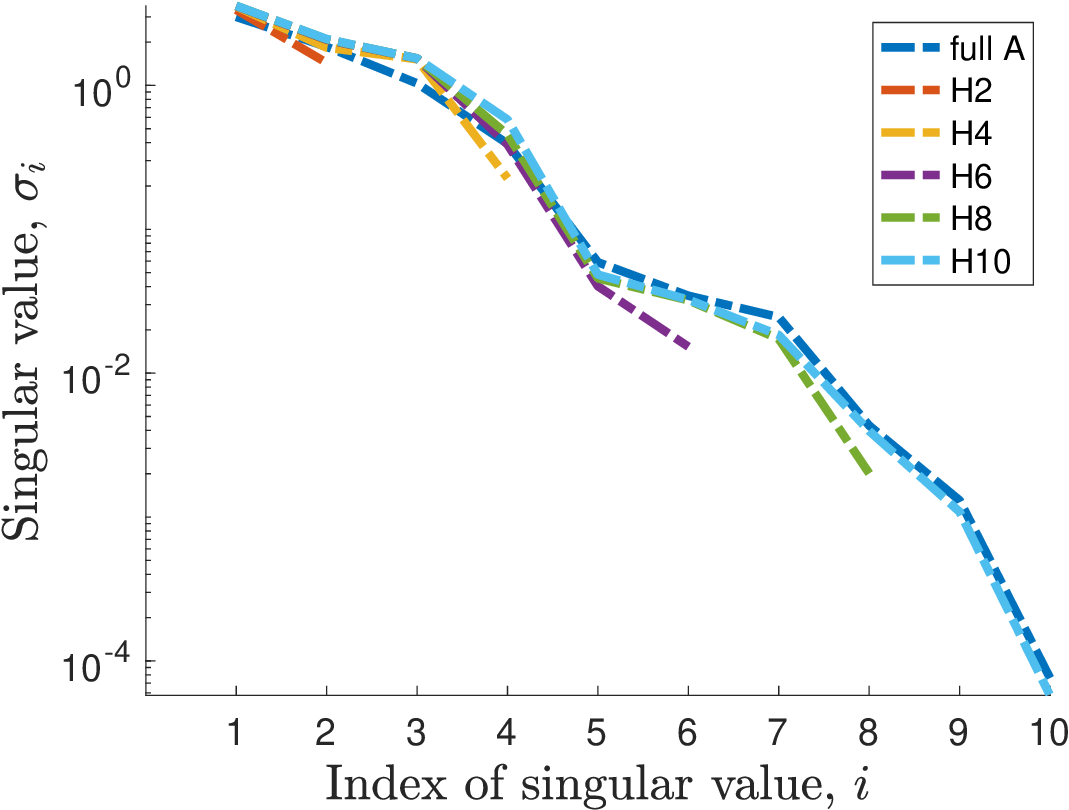} &
      \includegraphics[width=6cm]{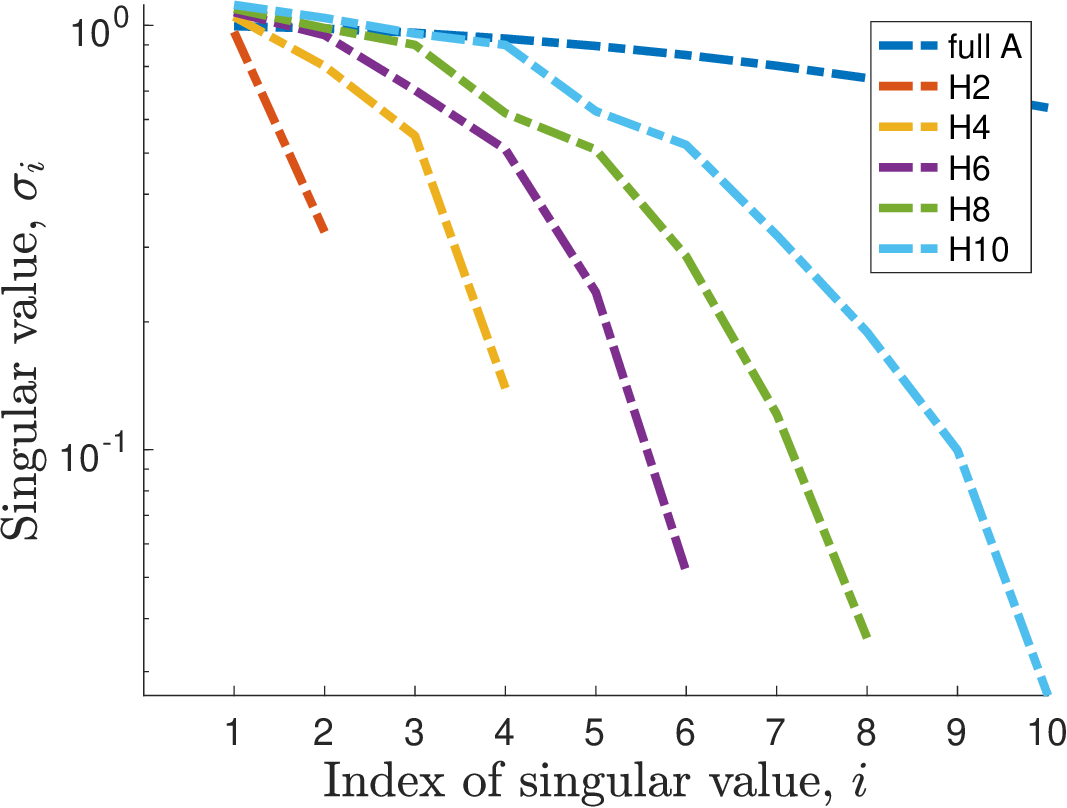}\\
\end{tabular}
\caption{Singular values for Shaw and Spectra.}
\label{fig:RegToolsSV2}
\end{figure}

From Figures \ref{fig:RegToolsSVD} and \ref{fig:RegToolsSV2}, it seems that the quality of the singular value approximations depends on the conditioning of $A$. To investigate this, we constructed the following experiment: consider the matrix $A \in \mathbb{R}^{n \times n}$ and the vector $x\true$ generated from Spectra test problem (\ref{eq:SpectraMatrix}). We compute the SVD of $A$ and keep the singular vectors from the orthogonal matrices $U$ and $V$ but replace the diagonal matrix containing the singular values of $A$ with a diagonal matrix $S$ whose diagonal elements are defined as: $$ S^{(i)}_{kk} = e^{c_i*k} \quad \mbox{for} \quad k = 1,...,n,$$ where $c_i$ are entries from the vector $c = [-2,-1,-0.5,-0.25]$.  For these different matrices $S^{(i)}$, we compare the largest singular values of $A^{(i)} = US^{(i)}V^T$ with those associated with the Hessenberg matrices obtained using the Hessenberg process associated to $A^{(i)}$ and right-hand side $b^{(i)} = A^{(i)}x\true$. The results are displayed in Figures \ref{fig:ContrivedSVD} and \ref{fig:ContrivedSVD2}. Indeed, it appears that the more ill-conditioned $A$ is, the better the singular values of the Hessenberg matrices associated to the Hessenberg process approximate the largest singular values of $A$. This result is coherent with what we observed for the previous experiments, for example in the test problems Shaw, in Figure \ref{fig:RegToolsSV2}, and Deriv2, in Figure \ref{fig:RegToolsSVD}. 

\begin{figure}[ht]
\begin{tabular}{cc}
    {\scriptsize Spectra at $c_i=-2$} &  {\scriptsize Spectra at $c_i=-1$} \\   
    \includegraphics[width=6cm]{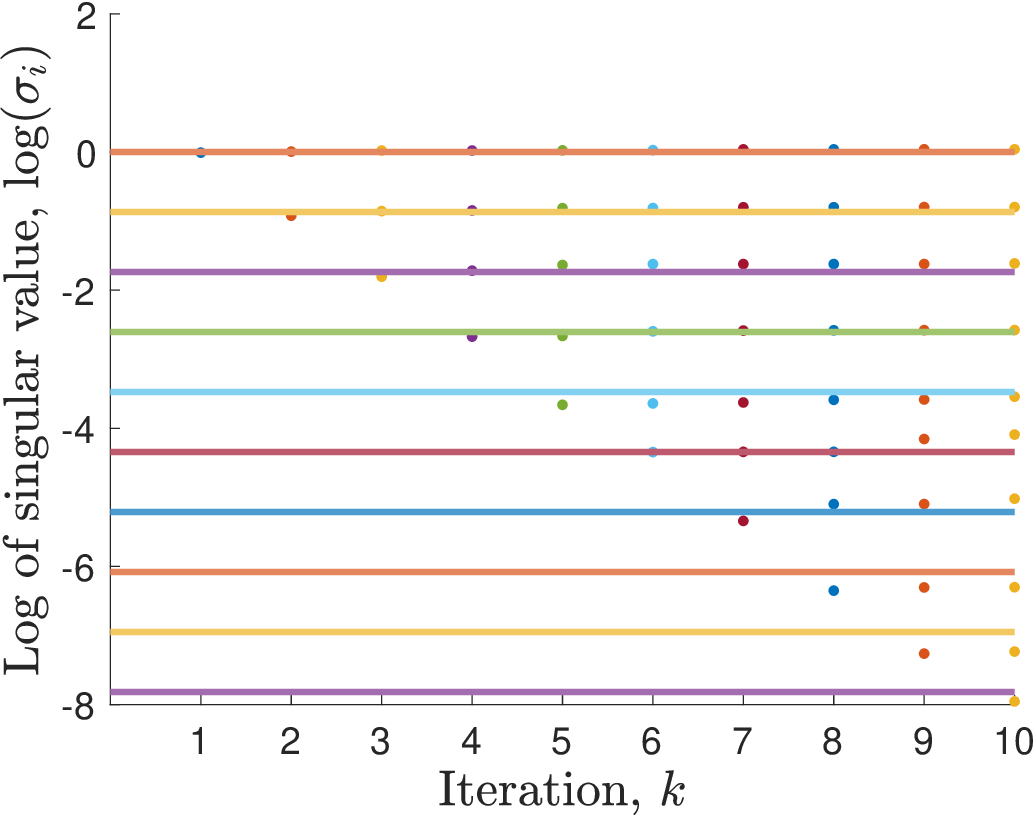} &  \includegraphics[width=6cm]{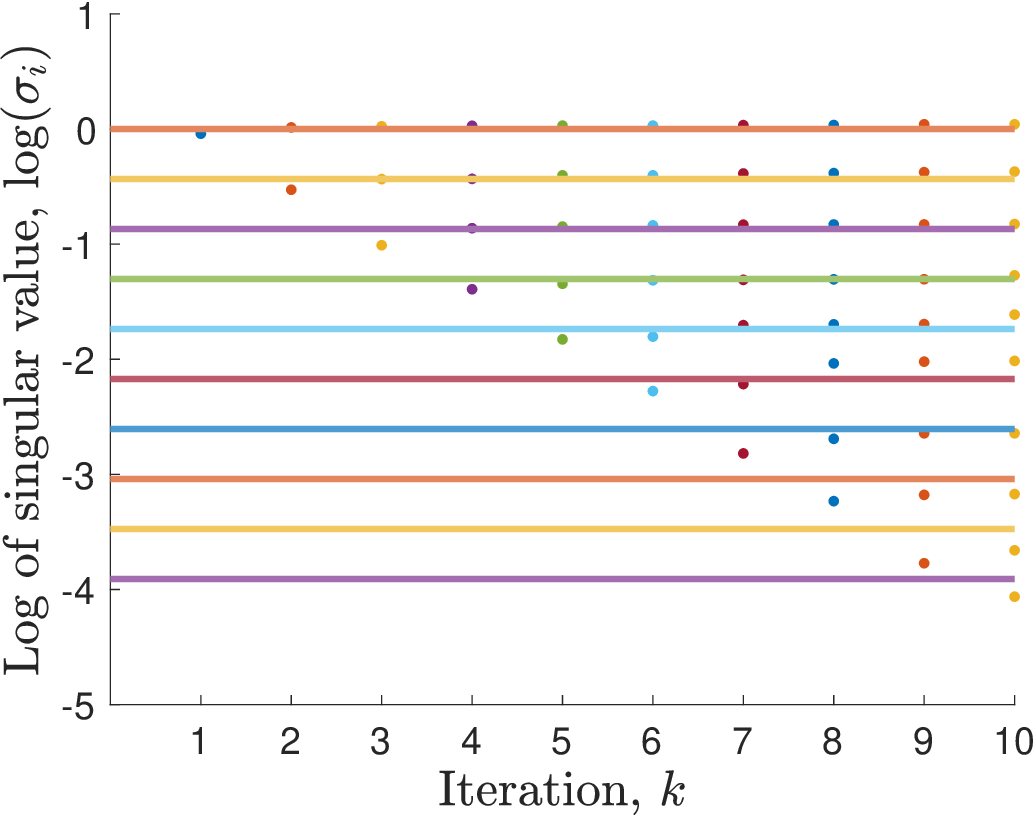} \\
      
\end{tabular} \\
\begin{tabular}{cc}
    \includegraphics[width=6cm]{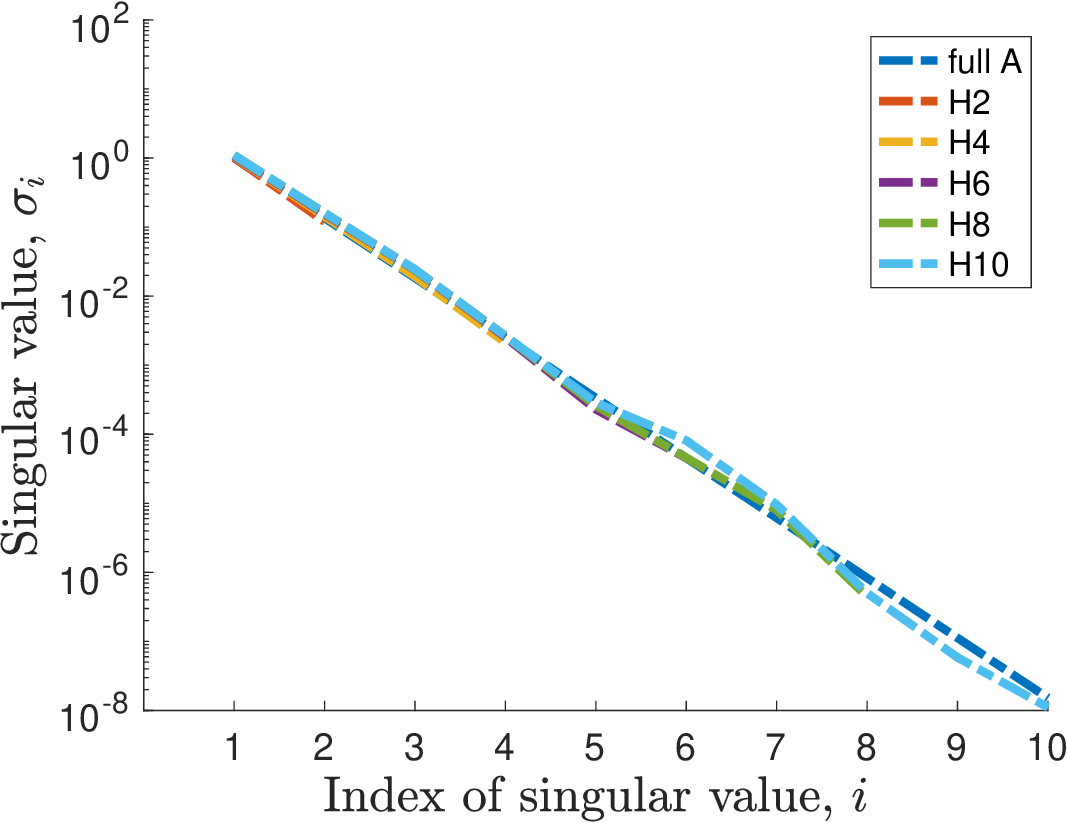} &  \includegraphics[width=6cm]{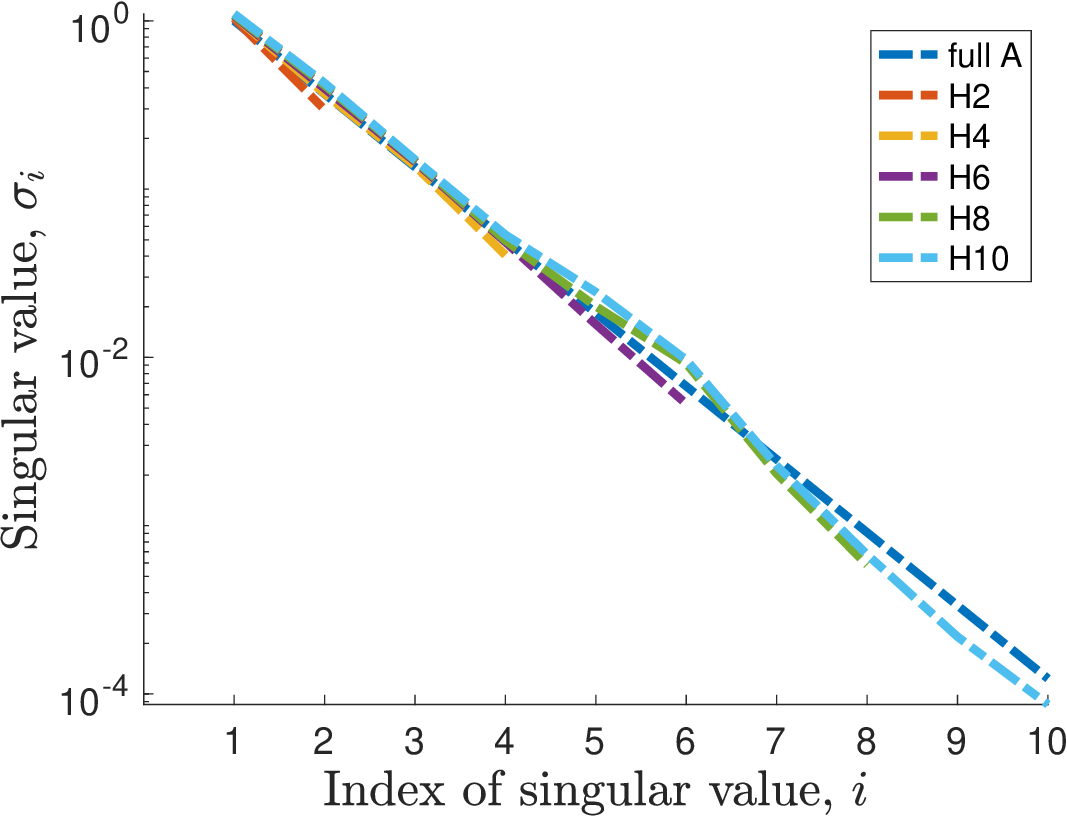} \\
\end{tabular}
\caption{Singular values for modified Spectra.}
\label{fig:ContrivedSVD}
\end{figure}

\begin{figure}[ht]
\begin{tabular}{cc}
    {\scriptsize Spectra at $c_i=-0.5$} &  {\scriptsize Spectra at $c_i=-0.25$} \\ 
    \includegraphics[width=6cm]{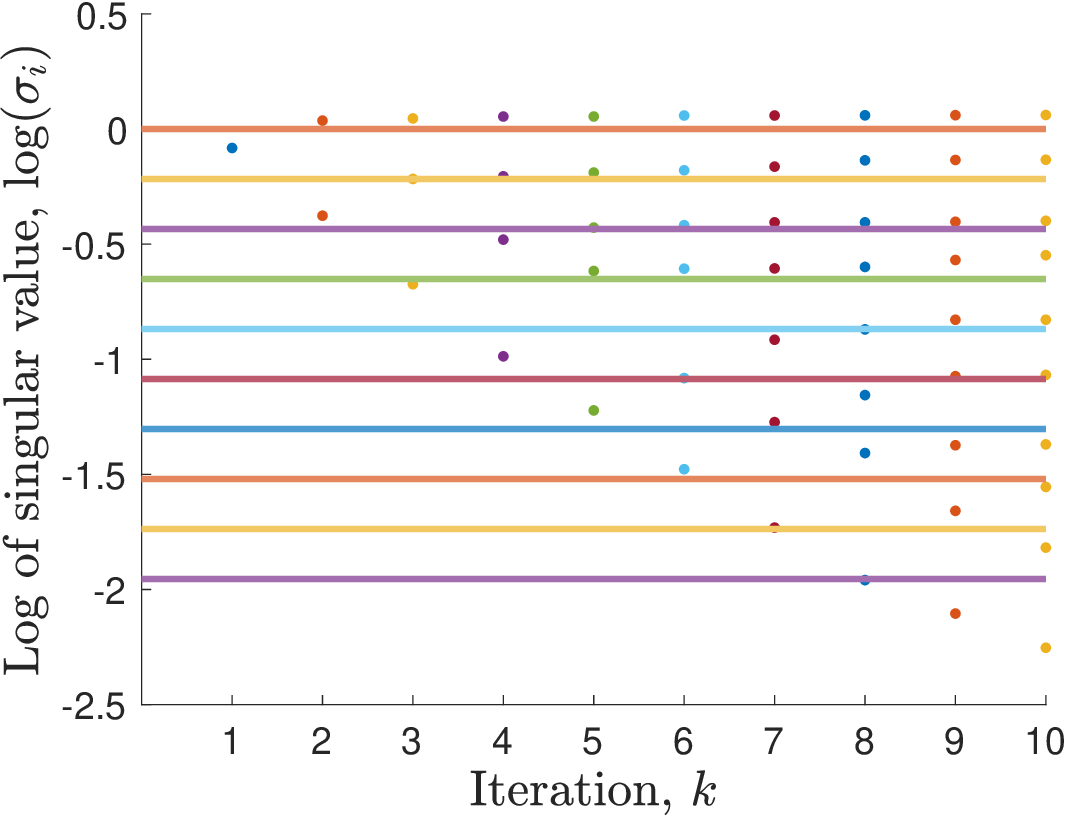} &  \includegraphics[width=6cm]{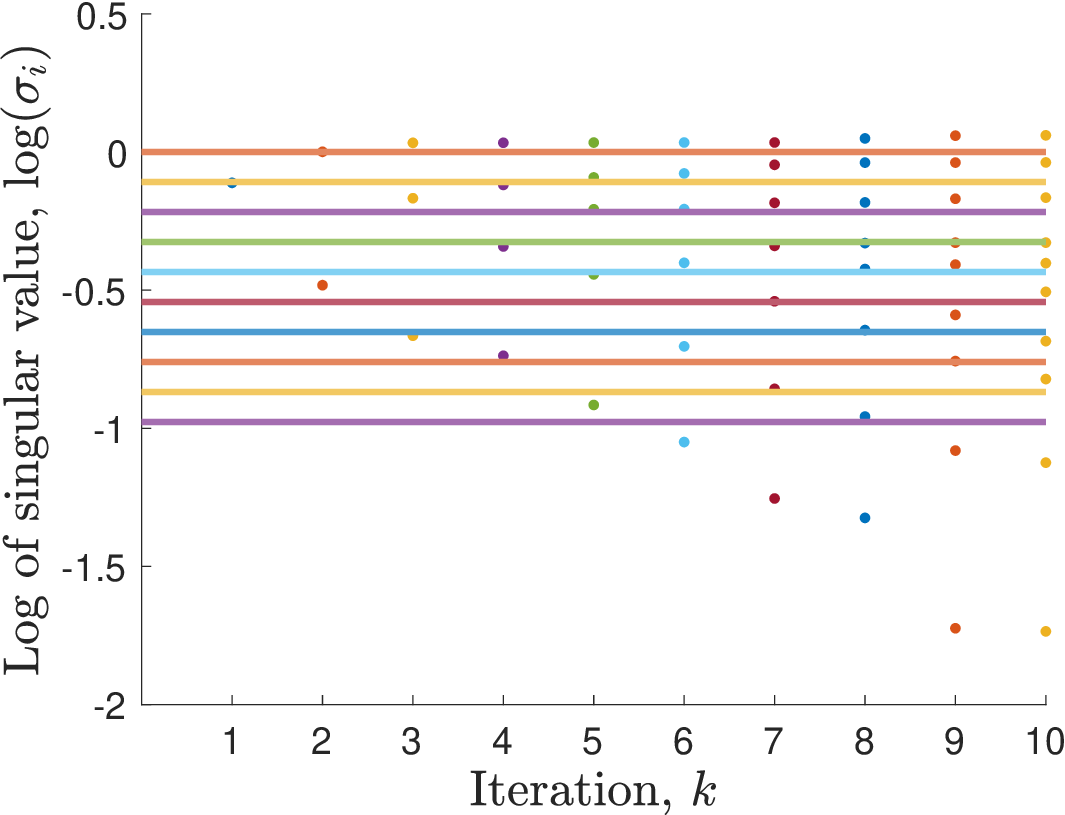} \\
\end{tabular} \\
\begin{tabular}{cc}
    \includegraphics[width=6cm]{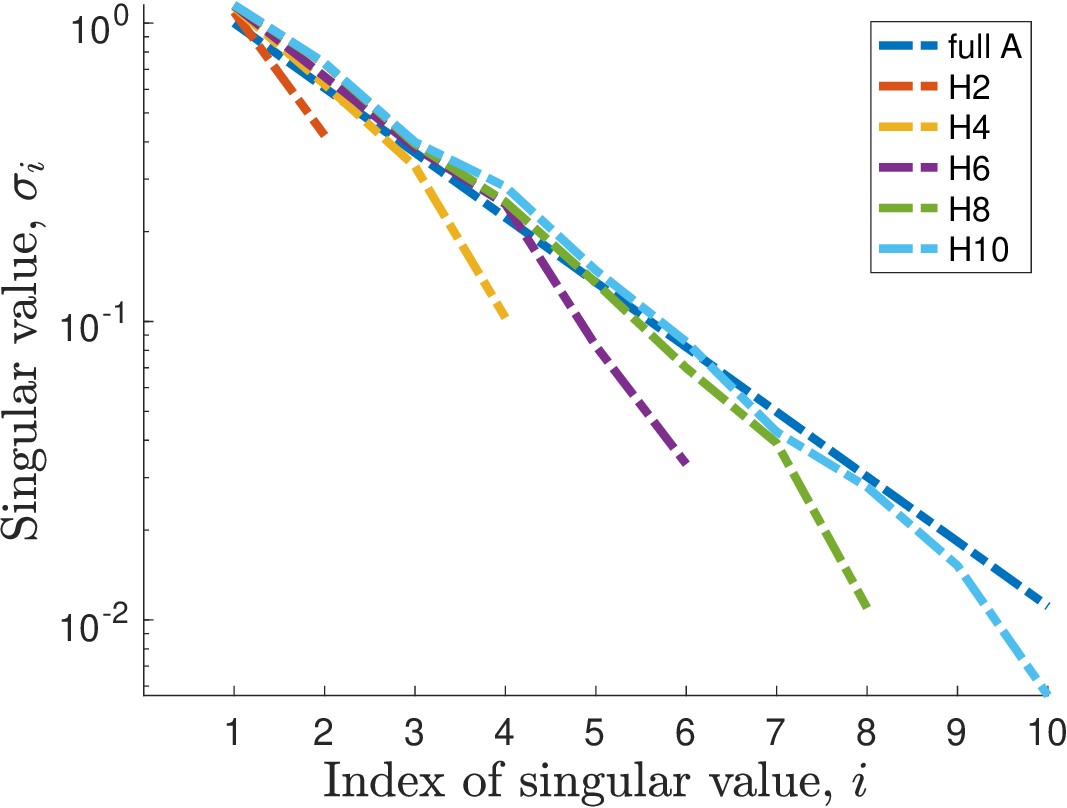} &  \includegraphics[width=6cm]{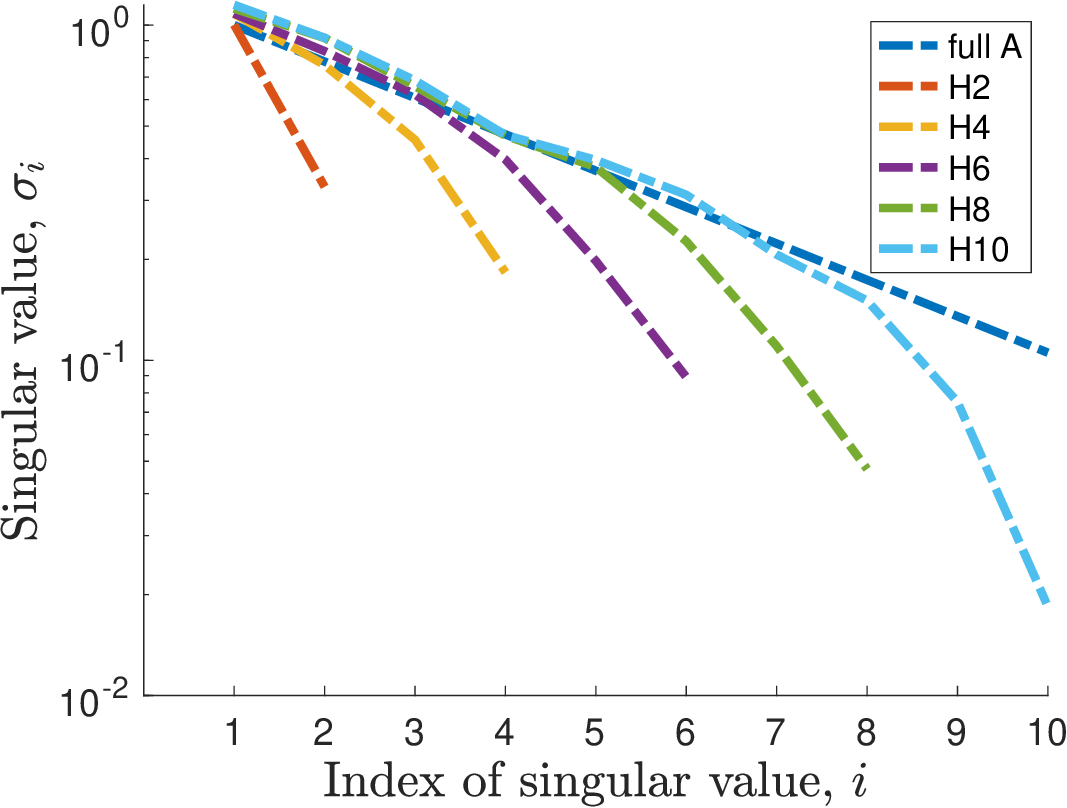} \\
\end{tabular}
\caption{Singular values for modified Spectra.}
\label{fig:ContrivedSVD2}
\end{figure}

\subsection{Spectral Filtering Properties}
An approximated solution to \eqref{eq:1} is said to be regularized using spectral filtering if it can be written in the form
\begin{equation}\label{eq:filter}
    x = \sum_{i=1}^{n} \Phi_i \frac{u_i^T b}{\sigma_i} v_i,
\end{equation}
where $u_i$ and $v_i$ correspond to left and right singular vectors of $A$, and $\sigma_i$ are its singular values in decreasing order of magnitude. Spectral filtering methods are a wide class of methods, including for example Tikhonov regularization, see, e.g. \cite{Hansen2010}. Moreover, note that  the truncated SVD corresponds to setting $\Phi_i=1$ for $i$ smaller than the truncation index and $\Phi_i=0$ otherwise.  There is also an empirical relation between Krylov methods and filtered solutions of the form \eqref{eq:filter}, see, e.g. \cite{Nagy2003SD}. 

A natural question arises when studying the regularizing properties of CMRH: is there an empirical relationship between the approximate solutions computed with CMRH and a solution with spectral filtering of the form in \eqref{eq:filter}? Using \eqref{eq:filter}, one can compute the empirical associated filter factors $\Phi_i$ to a given approximated solution at the $k$th iteration, $x_k$, as 
\begin{equation}\label{eq:phi}
    \Phi_i = \frac{v_i^T x_k }{u_i^T b}\sigma_i,
\end{equation}
assuming the SVD of $A$ is available. 

Consider the $1$D signal restoration problem Spectra solved using CMRH and GMRES. In Figure \ref{fig:FF_Enrm}, we observe the relative error norm \rev{histories} for three different noise levels. For each noise level, the plots in Figure \ref{fig:FF123} display the empirical filter factors for some iterations of CMRH and GMRES. These are \rev{indicated} 
with a marker in Figure \ref{fig:FF_Enrm}, \rev{and aim} to illustrate 
three iterations that are common for all noise levels and the last iteration before the relative error increases.

First, one can observe that for all combinations of noise levels and iterations displayed in the plots in Figure \ref{fig:FF123}, the empirical filter factors for CMRH and GMRES have a very similar behavior. This is also reminiscent of the filter factors corresponding to the TSVD, which \rev{present} 
a sharp phase transition between filter factors being 1 and 0. Also note that this phase transition moves to the right (that is, happens at a higher value of $k$) as the iterations increase: in other words, fewer iterations correspond to more filtering. This is a very positive result, as it indicates that CMRH is an effective method for filtering highly oscillatory noisy components if the iterations are stopped early. 

We also observe that the different noise levels (in this small regime) do not affect significantly the filter factors in the first iterations. We reiterate that in order for CMRH to mimic spectral filtering regularization, the iterative solver must terminate before the relative error increases.

\begin{figure}[ht]
\begin{tabular}{ccc}
    {\scriptsize Noise level 0.0001} &  {\scriptsize Noise level 0.001} & {\scriptsize Noise level 0.01} \\ 
    \includegraphics[width=3.85cm]{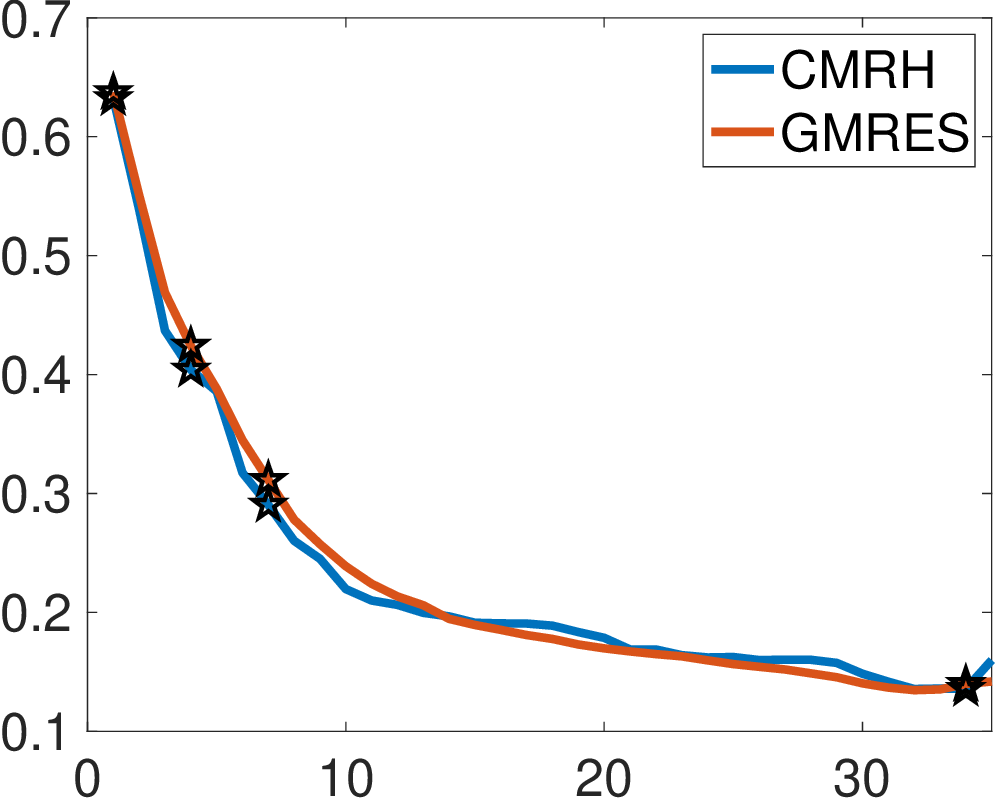} & 
    \includegraphics[width=3.85cm]{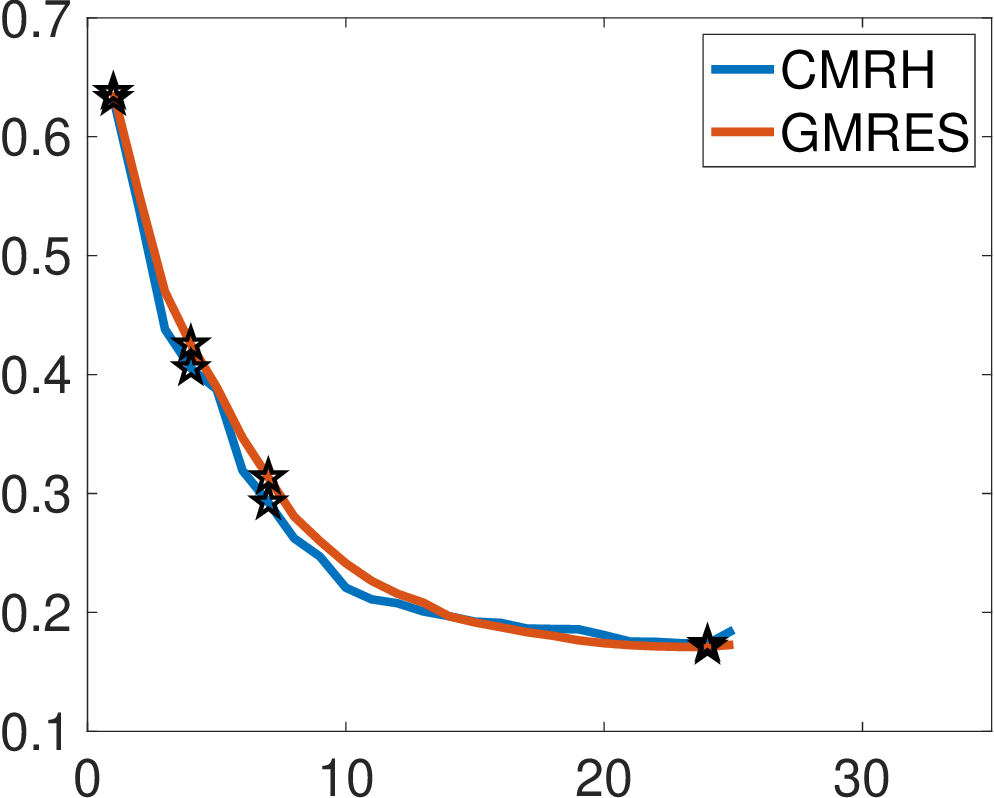} & \includegraphics[width=3.85cm]{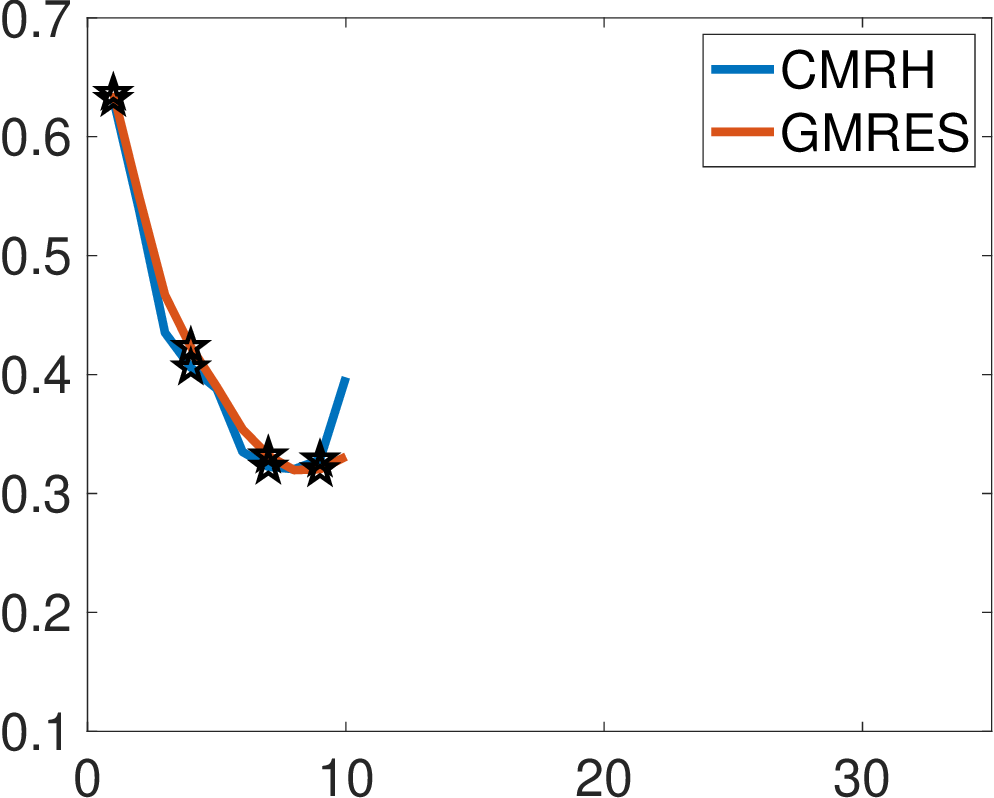} 
\end{tabular} 
    \caption{Relative error norms for the approximate solutions to test problem Spectra computed using CMRH and GMRES for different noise levels. The empirical filter factors for both methods at the indicated markers can be observed in the plots in Figure \ref{fig:FF123}.}
\label{fig:FF_Enrm}
\end{figure}
 and 
\begin{figure}[ht]
\begin{tabular}{cccc}
    {\scriptsize iteration $k$=1} &  {\scriptsize iteration $k$=4}  &{\scriptsize iteration $k$=7} &  {\scriptsize iteration $k$=9} \\ 
     \includegraphics[width=2.8cm]{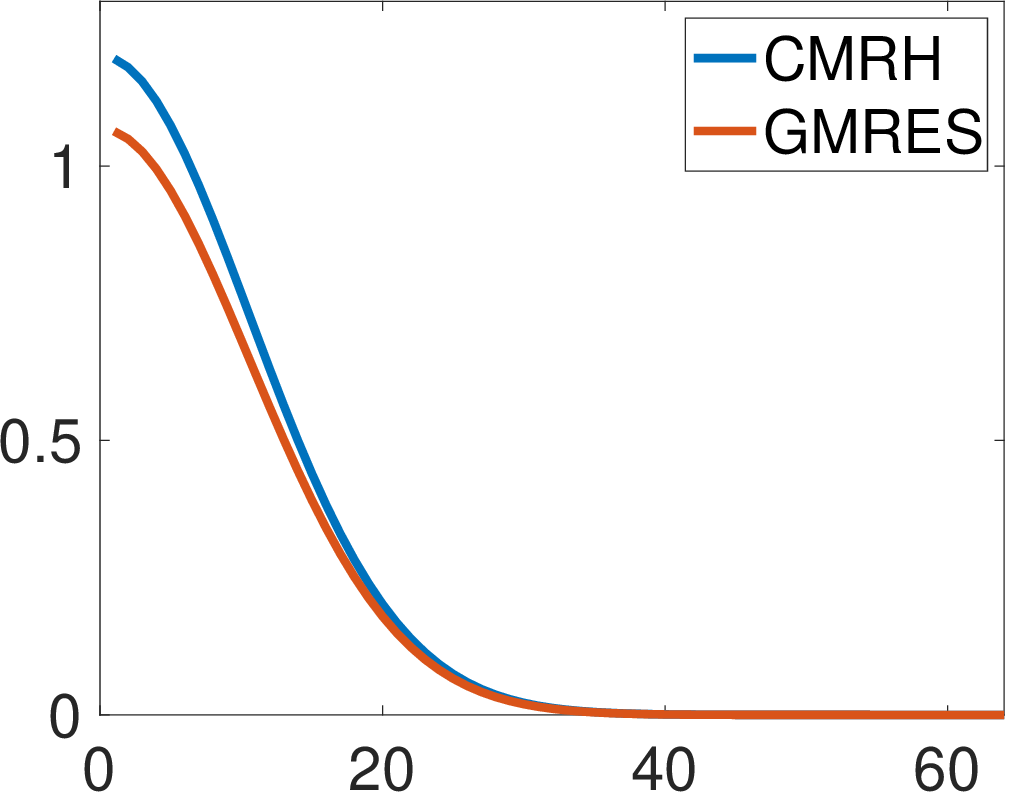} & 
    \includegraphics[width=2.8cm]{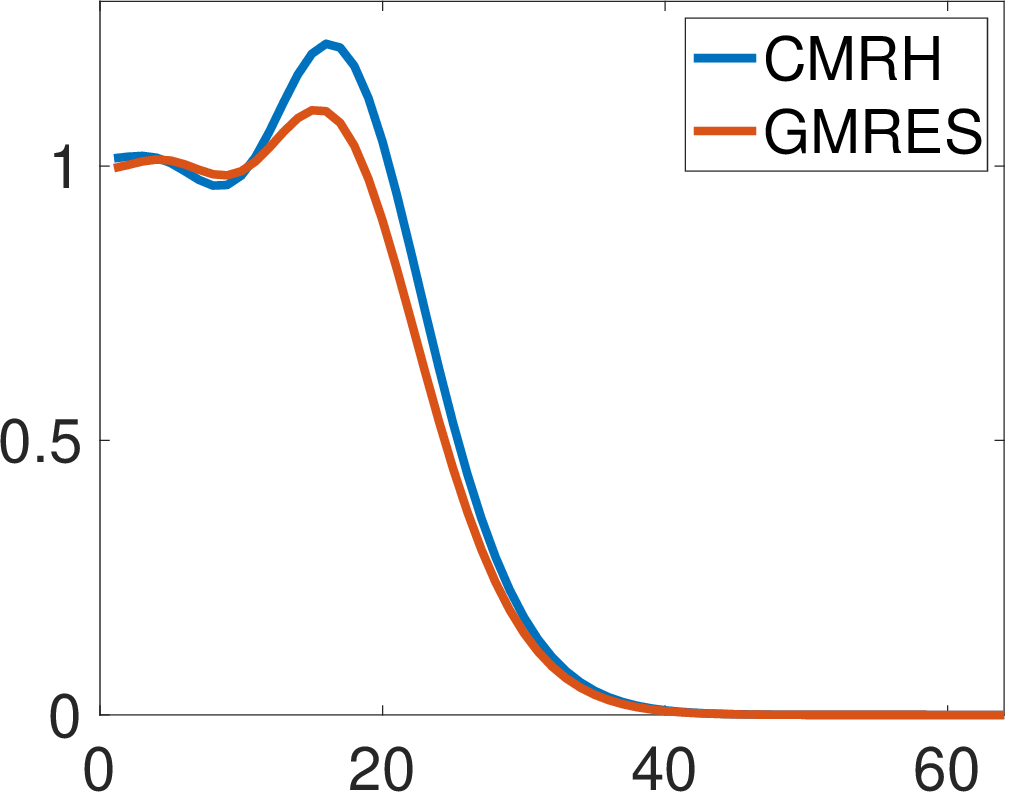} & \includegraphics[width=2.8cm]{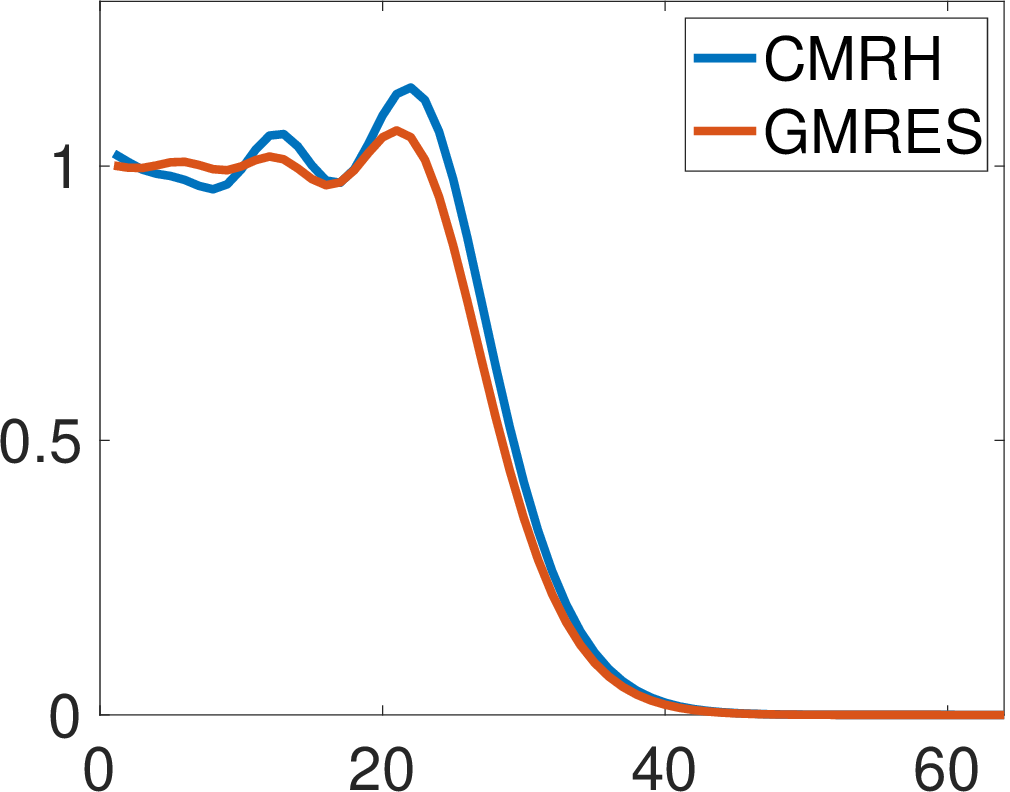} & \includegraphics[width=2.8cm]{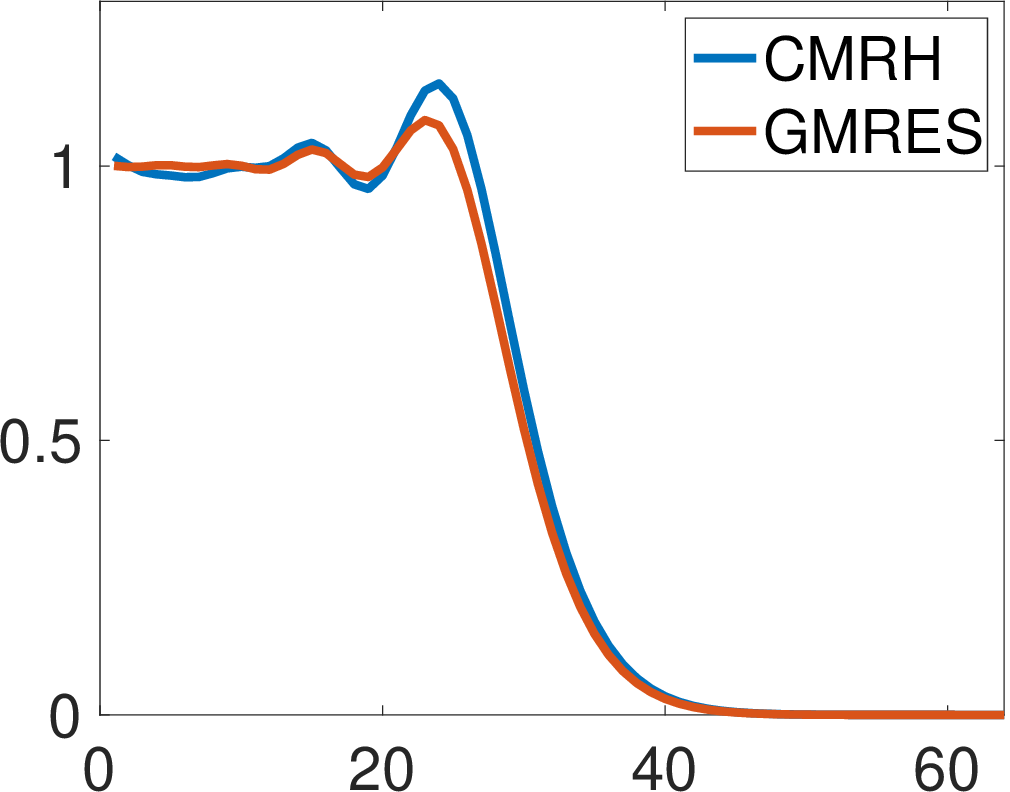} \\[12pt]

    {\scriptsize iteration $k$=1} &  {\scriptsize iteration $k$=4}  &{\scriptsize iteration $k$=7} &  {\scriptsize iteration $k$=24} \\ 
    \includegraphics[width=2.8cm]{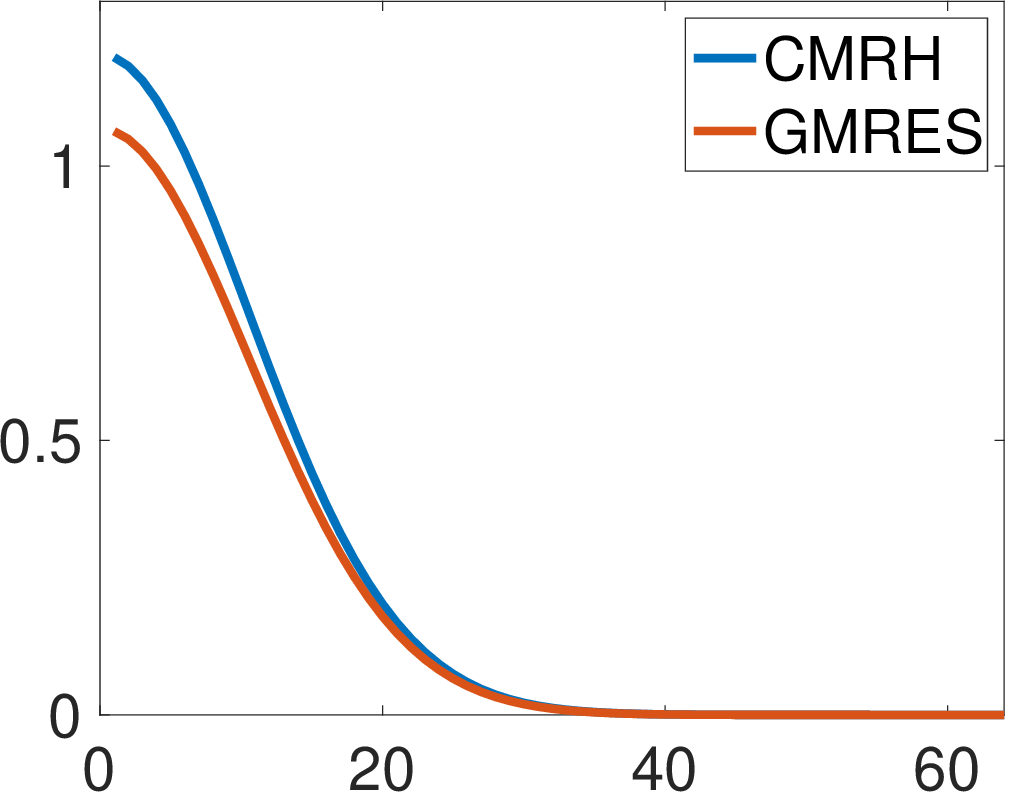} & 
    \includegraphics[width=2.8cm]{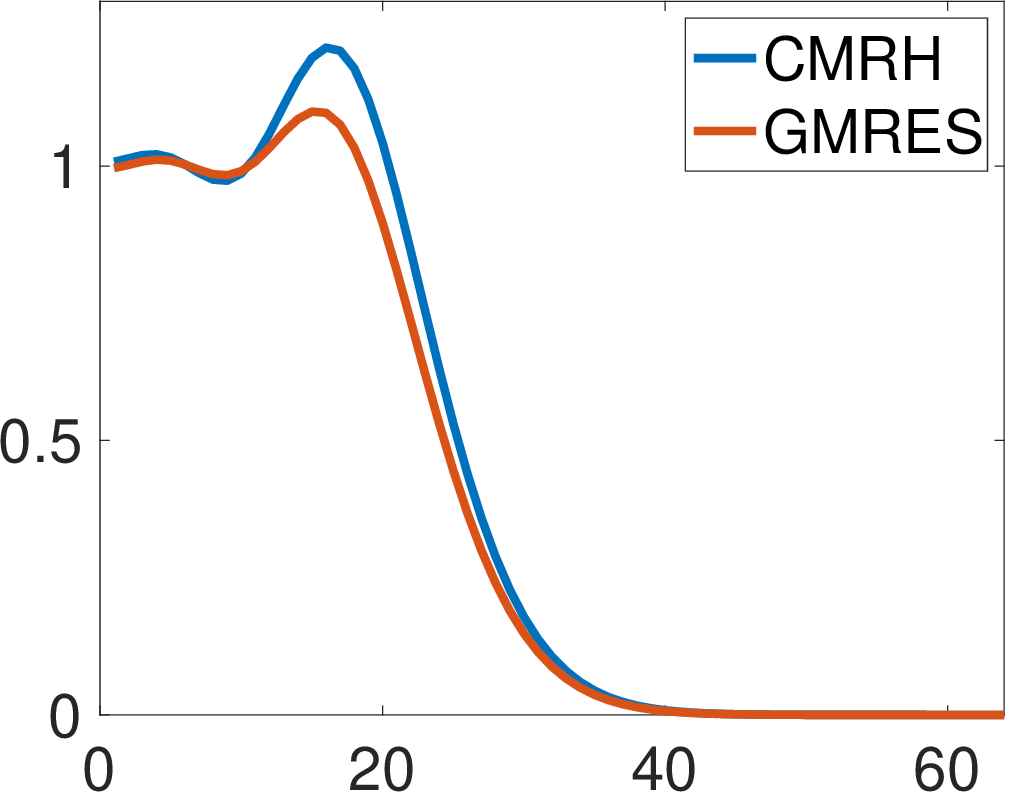} & \includegraphics[width=2.8cm]{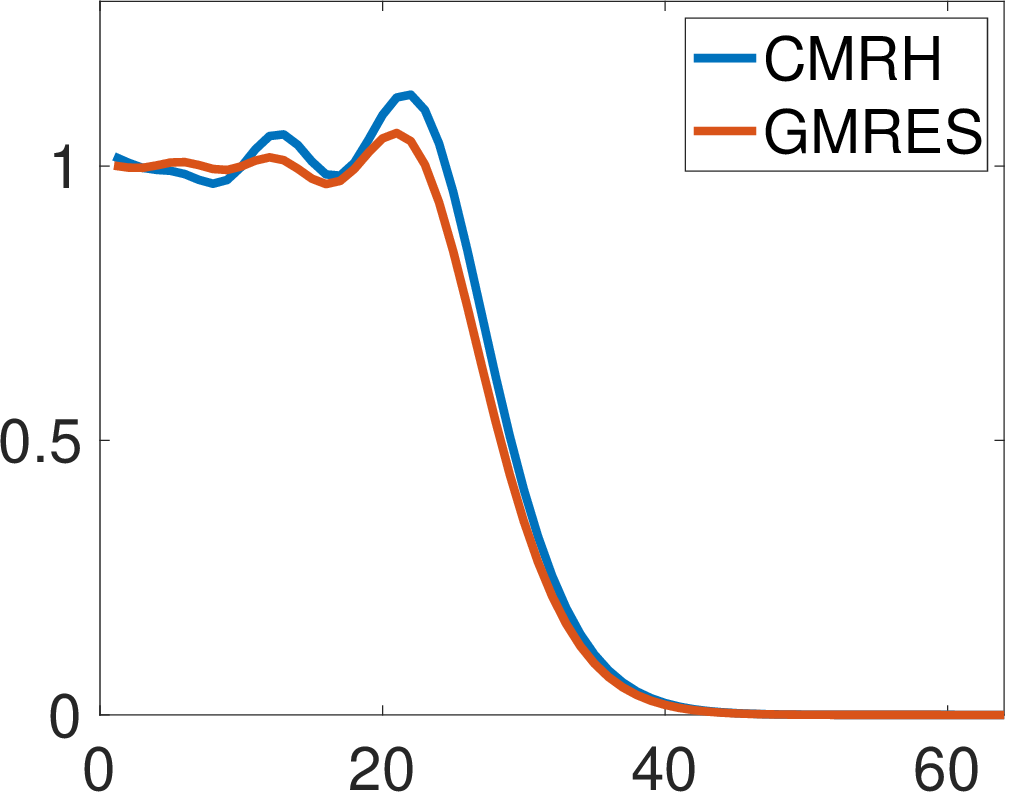} & \includegraphics[width=2.8cm]{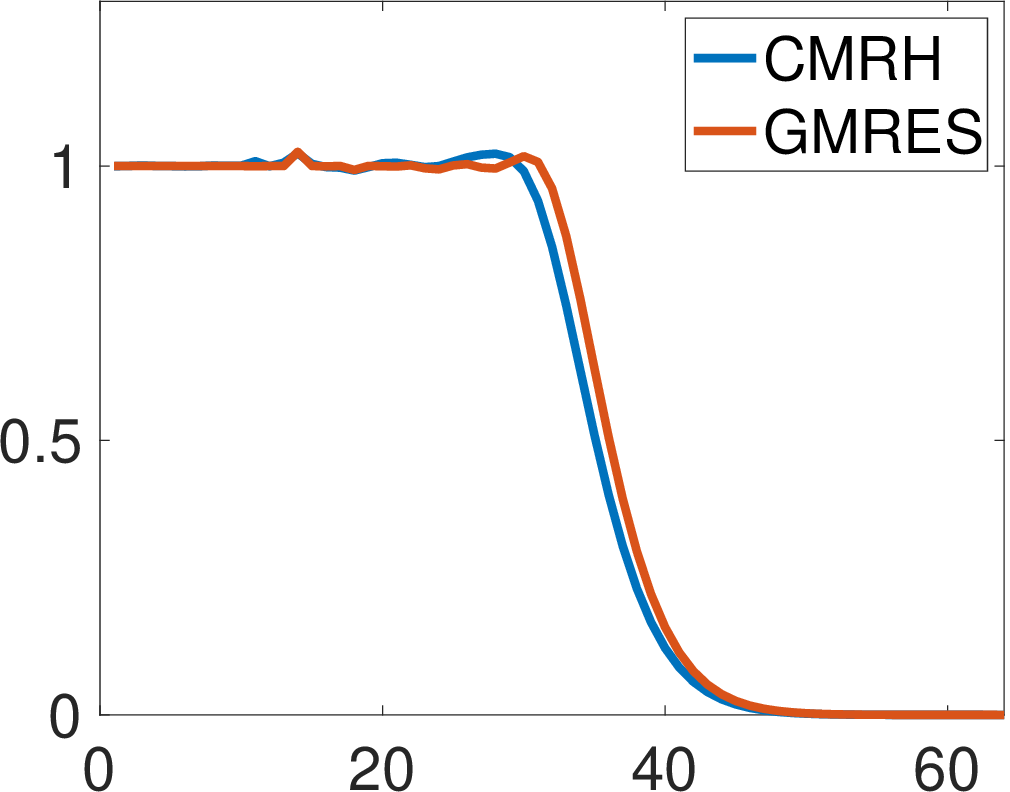} \\[12pt]

    {\scriptsize iteration $k$=1} &  {\scriptsize iteration $k$=4}  &{\scriptsize iteration $k$=7} &  {\scriptsize iteration $k$=34} \\ 
    \includegraphics[width=2.8cm]{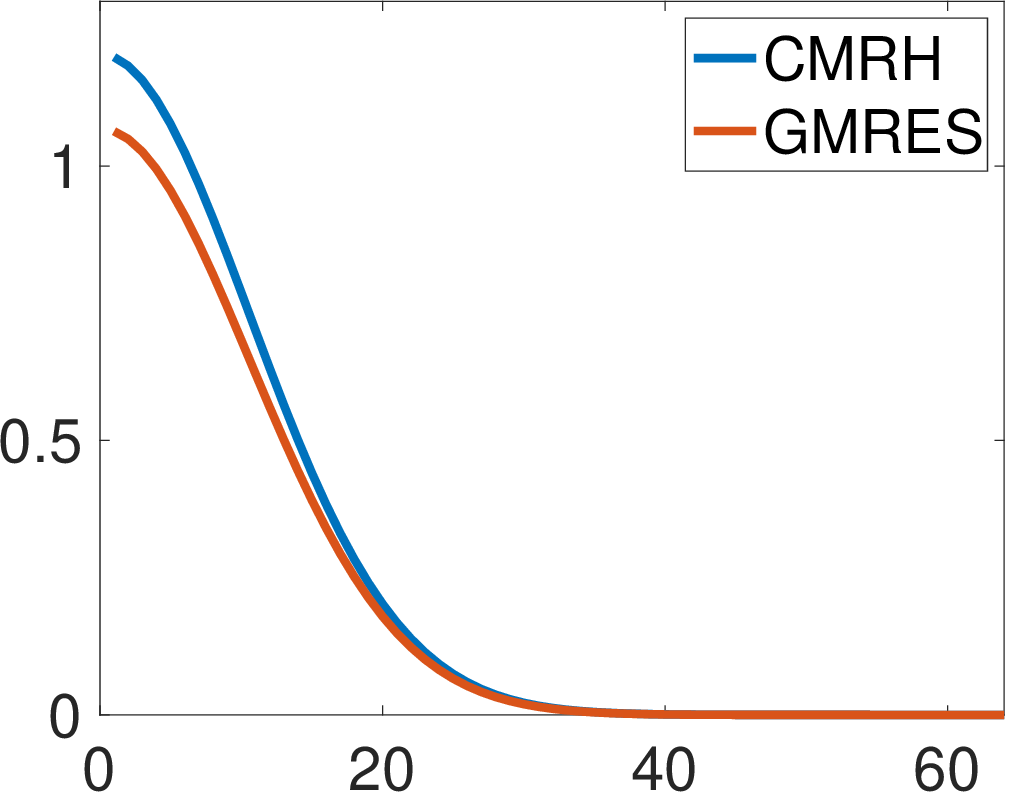} & 
    \includegraphics[width=2.8cm]{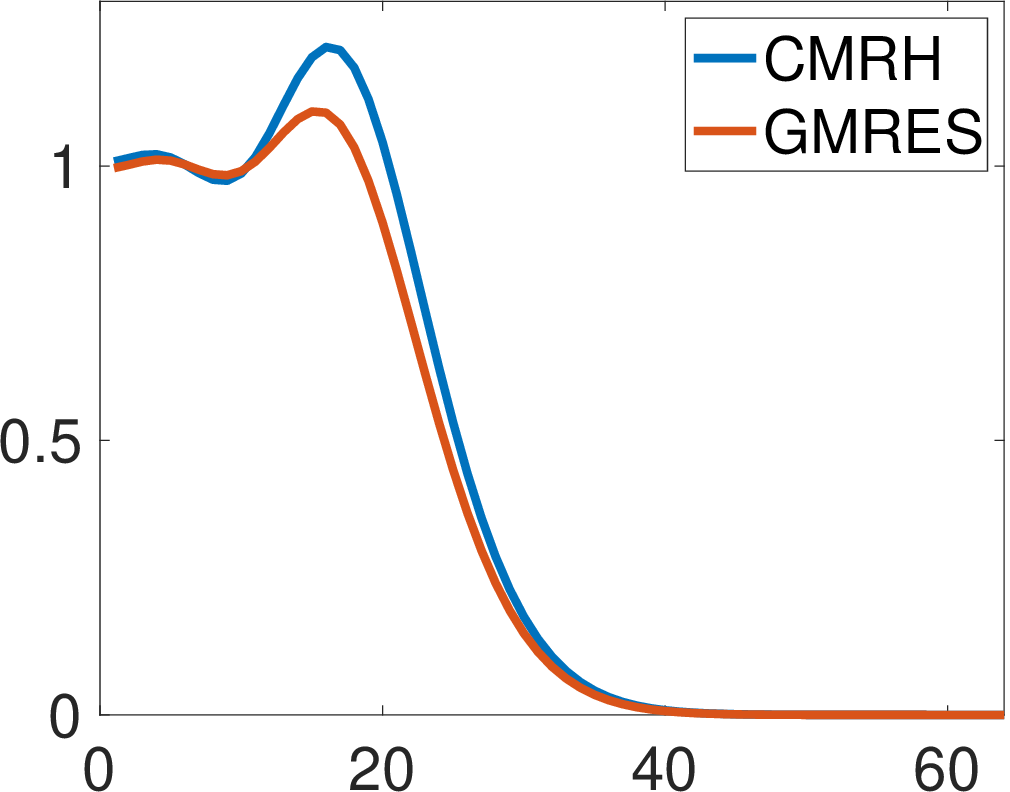} & \includegraphics[width=2.8cm]{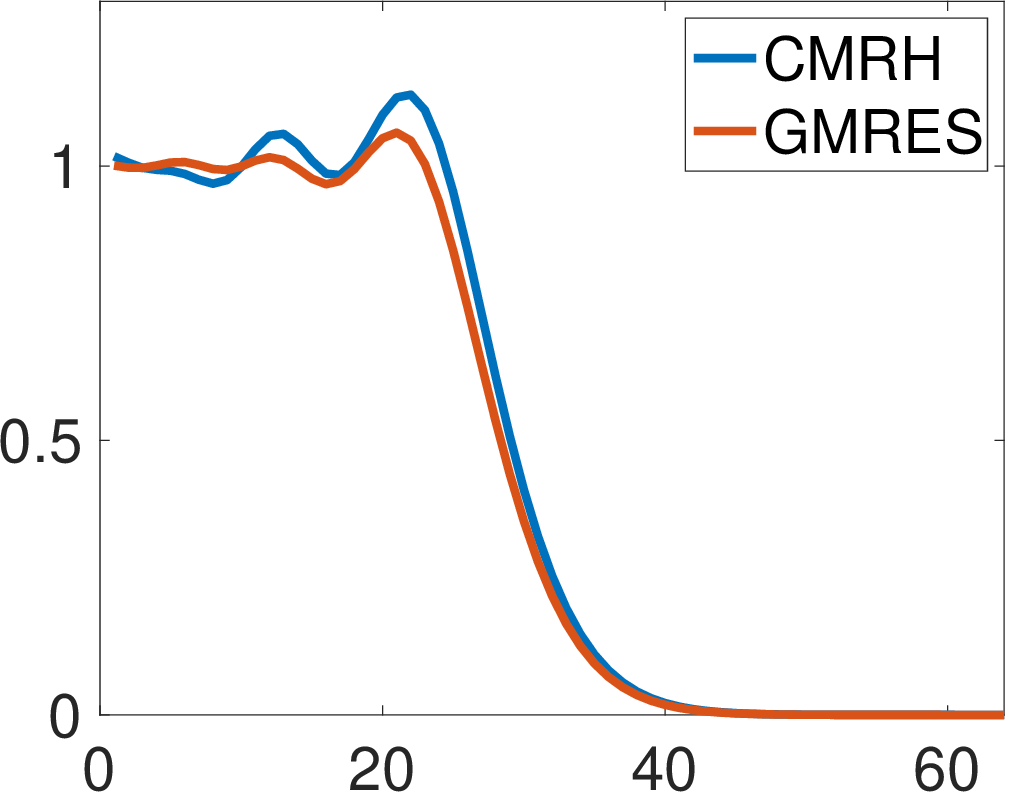} & \includegraphics[width=2.8cm]{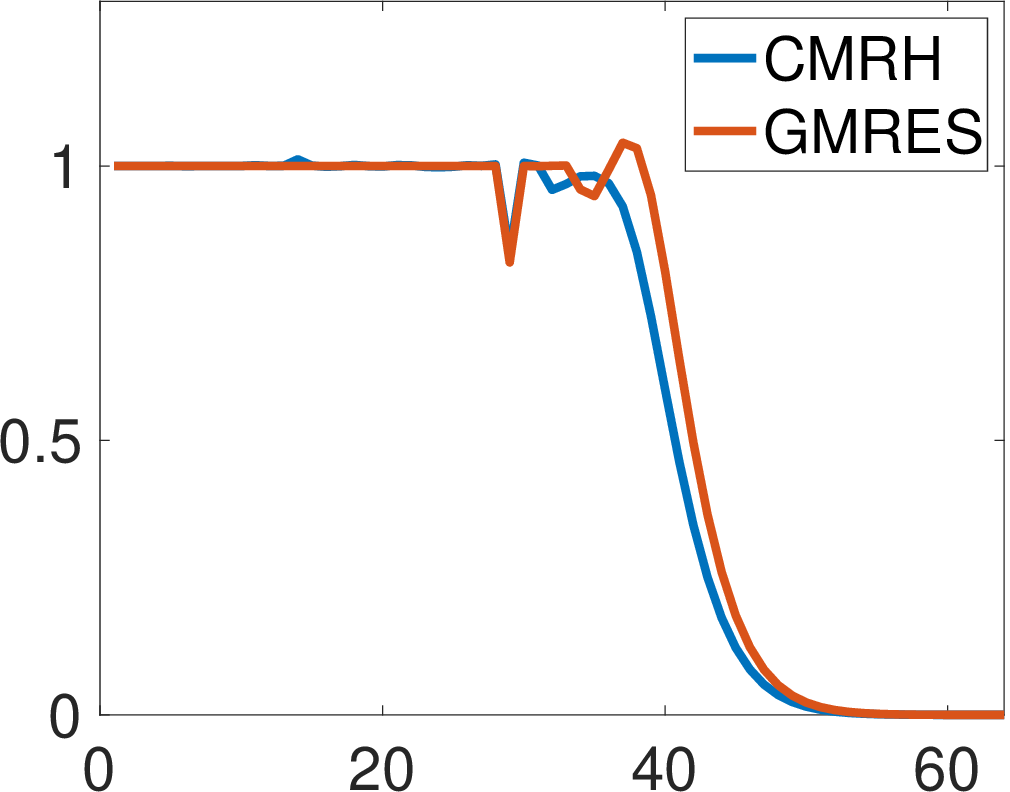} 
\end{tabular} 
\caption{
Empirical filter factors $\Phi_i$ for solutions of test problem Spectra computed using GMRES and CMRH at iterations corresponding to the markers in Figure \ref{fig:FF_Enrm}. 
The top row corresponds to left plot in Figure \ref{fig:FF_Enrm} (noise level 0.001), the middle row corresponds to the middle plot of Figure \ref{fig:FF_Enrm} (noise level 0.001), and the bottom row corresponds to the right plot in Figure \ref{fig:FF_Enrm} (noise level 0.01).
}
\label{fig:FF123}
\end{figure}
\section{Hybrid Changing Minimal Residual Hessenberg Method}\label{sec:H-CMRH}
Consider the following optimization problem with standard Tikhonov regularization: 
\begin{equation} \label{eq:17} 
\min_{x \in \mathbb{R}^n} \|b-Ax\|^{2} + \lambda^2 \|x \|^{2}. 
\end{equation}
At each iteration, we consider an approximate solution belonging to the Krylov subspace ${\cal{K}}_k$. Due to the lack of orthonormality of the basis vectors generated by the Hessenberg process, we solve the following least squares problem:
\begin{equation} \label{eq:18} 
\min_{x \in K_{k}} \| L^{\dagger}_{k+1} (b-Ax) \|^{2} + \lambda^2_{k} \| L^{\dagger}_{k} x \|^{2},
\end{equation}
which is an approximation of (\ref{eq:17}).
Similarly to CMRH, and using the Hessenberg relationship in (\ref{eq:7}), solving \eqref{eq:18} is equivalent to solving the following small subproblem after each iteration of the Hessenberg process:
\begin{equation} 
\label{eq:19} 
y_{\lambda , k} = \arg \min_{y \in \mathbb{R}^{k}} \| \beta e_{1} - H_{k+1,k}y \|^{2} + \lambda^2_{k} \| y \|^{2}, 
\end{equation}
where $\beta$ is the largest entry in $r_0$\rev{, 
assuming pivoting is used}, and then project\rev{ing} back to the original subspace using $x_k = x_0 + L_k y_{\lambda , k}$ as seen in Algorithm \ref{alg:arbitrarySample3}. 
Note that one of the main advantages of hybrid methods is that the regularization parameter $\lambda$ in \eqref{eq:17} can be adaptively chosen throughout the iterations: hence the subscript $\lambda_k$ in \eqref{eq:18}. In our case, this is due to the nature of the regularization term in \eqref{eq:17}, which implies that the projection onto the Krylov subspace is independent of the regularization parameter. Our goal is to provide an appropriate $\lambda_k$ at each iteration that will not cause the regularized solution $x_k$ to be overly smooth or oscillatory. 

There exist many regularization parameter choice criteria for the Tikhonov problem, see, e.g. \cite{Kilmer2001RegParam, Reichel2013Param, Bauer2011Param}, and more specifically for hybrid regularization, see, e.g. \cite{chung2024computational,Gazzola2020Krylov}. Moreover, since the projected problem \eqref{eq:19} is of a small dimension, SVD based approaches of the projected matrix can be implemented to find $\lambda_{k}$ very efficiently. In this paper, we analyze the results using the optimal regularization parameter with respect to the error norm 
and the generalized cross validation (GCV) method. These are explained in detail in Section \ref{sec:RegParam}. However, note that other regularization parameter criteria can be adapted seamlessly to this framework.

A full description of the pseudocode for H-CMRH with pivoting can be found in Algorithm \ref{alg:arbitrarySample3}.

\begin{algorithm}[H]
\caption{Hybrid CMRH (H-CMRH)} \label{alg:arbitrarySample3}
\begin{algorithmic}[1]
\State Compute $r_0 = b -Ax_0$
\State $p = [1,2,....,n]^T$
\State Determine $i_0$ such that $|r_0(i_0)| = \|r_0\|_{\infty}$
\State $\beta = r_0(i_0)$; $l_1 = r_0/ \beta$ ; $p_1 \Leftrightarrow p_{i_0}$
\For{$k = 1,....,m$}
\State $ u =Al_k$
\For{$j = 1,....,k$}
\State $h_{j,k} = u(p(j))$; $u = u - h_{j,k}l_j$
\EndFor
\If {$k<n$ and $ u \neq 0$}
\State Determine $i_0 \in \{ k+1,.......,n \}$ such that $|u(p(i_0))| = \|u(p(k+1:n))\|_{\infty}$
\State $h_{k+1,k} = u(p(i_0))$; $l_{k+1} = u/h_{k+1,k}$ ; $p_{k+1} \Leftrightarrow p_{i_0}$
\Else
\State $h_{k+1,k} = 0$; Stop.
\EndIf
\EndFor
\State Define the $(k+1) \times k$ Hessenberg matrix 
\State Implement regularization parameter scheme ($\lambda_k$): optimal or GCV
\State Compute the minimizer of $ \| \beta e_{1} - H_{k+1,k}y \|^{2} + \lambda^2_{k} \| y \|^{2}$ and $x_k = x_0 + L_k y_{\lambda,k}$
\end{algorithmic} 
\end{algorithm} 

\subsection{Residual Norms for Hybrid CMRH and Hybrid GMRES}
The difference between the residual norms of H-CMRH and hybrid GMRES can be bounded in a similar fashion to those of CMRH and GMRES. Since the validity of hybrid GMRES is well established, these bounds are crucial to understand the regularizing properties of H-CMRH. If we define 
\begin{equation}\label{eq:Lbar}
    \overline{L}_{k+1}= \begin{bmatrix} L_{k+1} & 0 \\ 0 & L_k  
    \end{bmatrix},
\end{equation}
with $L_{k},L_{k+1}$ defined by the Hessenberg relation \eqref{eq:7}, we find that if the condition number of $\overline{L}_{k+1}$  does not grow too quickly, then the residual norm associated to the solution obtained with H-CMRH is close \rev{to} the residual norm of the solution obtained with hybrid GMRES. 
\begin{theorem}
    Let $hr_k^G$ and $hr_k^C$ be the hybrid GMRES and hybrid CMRH residuals at the kth iteration beginning with the same initial residual $r_0$, respectively. Then \begin{equation} \label{eq:20} \|hr_k^G \| \leq \|hr_k^C \| \leq \kappa (\overline{L}_{k+1}) \|hr_k^G\|,\end{equation}
    where $\kappa (\overline{L}_{k+1}) = \|\overline{L}_{k+1}\| \|\overline{L}_{k+1}^{\dagger}\|$ is the condition number of $\overline{L}_{k+1}$.
\end{theorem}

\begin{proof}
    First, we prove the left inequality in (\ref{eq:20}).
    We can define the hybrid residual as a function of the solution:$$ hr(x) = \begin{bmatrix} b \\ 0
     \end{bmatrix}  -  \begin{bmatrix} A \\ \lambda I
     \end{bmatrix} x .$$ 
    Since $x_k^G$ and $x_k^C$ are in the Krylov subspace ${\cal{K}}_k$, by the optimality conditions of hybrid GMRES: 
    \begin{equation*}
       \|hr_k^G\| =\min_{x \in {\cal{K}}_k} \| hr(x) \| \leq \|hr(x_k^C)\| = \| hr_k^C \|. 
    \end{equation*}
    Hence $\|hr_k^G\| \leq \|hr_k^C\|$.

    Now we prove the right inequality in (\ref{eq:20}). Since for any $x \in {\cal{K}}_{k}$ we have that $Ax$$-$$b \in {\cal{K}}_{k+1}$$=$range($L_{k+1}$); and $x_k^G, x_k^C\in {\cal{K}}_{k}$= range($L_{k}$), we can write $hr_k^C$ and $hr_k^G$ as a linear combination of the columns of $\overline{L}_{k+1}$ defined in \eqref{eq:Lbar}. Let  $hr_k^C = \overline{L}_{k+1} u_k^C$ and $hr_k^G = \overline{L}_{k+1} u_k^G$. This implies that $u_k^C = \overline{L}_{k+1}^{\dagger} hr_k^C$ and $u_k^G = \overline{L}_{k+1}^{\dagger} hr_k^G$. Hence, $\| \overline{L}^{\dagger}_{k+1} hr_k^C \| = \| u_k^C \|$ and $\| \overline{L}^{\dagger}_{k+1} hr_k^G \| = \| u_k^G \|$.  By the optimality conditions of H-CMRH:
\begin{equation} \label{eq:pie2} \|u_k^C\| =  \min_{x \in {\cal{K}}_k}\left\| \begin{bmatrix} L^{\dagger}_{k+1} & 0 \\ 0 & L^{\dagger}_k  
    \end{bmatrix} \left( \begin{bmatrix}
        b \\ 0
    \end{bmatrix} - \begin{bmatrix}
        A \\ \lambda I
    \end{bmatrix} x \right) \right \| = \min_{x \in {\cal{K}}_k} \| \overline{L}^{\dagger}_{k+1} hr(x)\| .\end{equation}
Using (\ref{eq:pie2}) and the fact that $x_k^G$ is in ${\cal{K}}_{k}$ then $\|u_k^C\| \leq \|u_k^G\|$. Thus
$$ \|u_k^C \| \leq \| u_k^G \| = \| \overline{L}_{k+1}^{\dagger}hr_k^G \| \leq \| \overline{L}_{k+1}^{\dagger}\| \|hr_k^G\|. $$ On the other hand,
$$ \|hr_k^C \| = \|\overline{L}_{k+1}u_k^C \| \leq \| \overline{L}_{k+1}\| \|u_k^C\|.$$
Putting the above inequalities together gives the following relation:
\begin{eqnarray*}
  \| hr_k^C \| 
 &=& \| \overline{L}_{k+1} u_k^C \| \\
  &\leq&  \| \overline{L}_{k+1} \|  \|u_k^C\| \\
    &\leq&  \| \overline{L}_{k+1} \| \| \overline{L}^{\dagger}_{k+1}\| \|hr_k^G \| \\
    &=&\kappa (\overline{L}_{k+1}) \|hr_k^G\|\,,
 \end{eqnarray*}
so we conclude that $\|hr_k^G\| \leq \|hr_k^C\| \leq \kappa (\overline{L}_{k+1}) \|hr_k^G\|$
\end{proof}

\subsection{ Regularization Parameters}\label{sec:RegParam}
To validate the potential of the H-CMRH method independently of the parameter choice, we use the optimal regularization parameter with respect to the error norm. This means that we use the following minimizer:
\begin{equation} \label{eq:opt} 
\lambda_k = \arg\min_{\lambda} \| x_{\lambda , k} - x\true \|^{2}, 
\end{equation}
where $x_{\lambda , k}$ is the approximated solution at each iteration as a function of $\lambda$. Let $x_{\lambda, k} = x_0 + L_k y_{\lambda , k}$. Then the minimization in \eqref{eq:opt} can be written as
\begin{equation} \label{eq:21} 
\min_{\lambda} \| x_{\lambda , k} - x\true \|^{2} = \min_{\lambda} \| (x_0 + L_k y_{\lambda,k}) - x\true \|^{2}. 
\end{equation}
We use the normal equations of \eqref{eq:19} to find a closed form expression for $y_{\lambda,k}$, so that \eqref{eq:21} is equivalent to
\begin{equation} \label{eq:22} 
\min_{\lambda} \| (x_0 + (L_k(H_{k+1,k}^T H_{k+1,k} + \lambda ^2 I)^{-1} H_{k+1,k}^T \beta e_1)) - x\true\|^{2} 
\end{equation}
and using the SVD of $\displaystyle  H_{k+1,k} = U_k \Sigma_k V_k^T $ this can be further simplified to:
\begin{equation}\label{eq:23}
   \min_{\lambda} \| (x_0 + L_k((V_k \Sigma_k ^T \Sigma_k V_k^T + \lambda ^2 I)^{-1} V_k \Sigma_k ^{T} U_k^{T} \beta e_1) )- x\true\|^{2}.
\end{equation}
or, equivalently,
\begin{equation}\label{eq:24}
\min_{\lambda} \|x_{\lambda,k} - x\true\|^{2} = \min_{\lambda}  \| (x_0 + L_k V_k( \Sigma_k ^T \Sigma_k + \lambda^2 I)^{-1}  \Sigma_k ^{T} U_k^{T} \beta e_1) - x\true\|^{2}.
\end{equation}
This is, of course, not a reasonable regularization parameter choice criterion in practice, since it requires the knowledge of the true solution.

In a more realistic scenario, generalized cross validation (GCV) can be used to choose a suitable estimate for the regularization parameter $\lambda_k$ at each iteration. This scheme is a predictive statistics-based approach where prior estimates of the error norm are not needed \cite{chung2008weighted}. Here, it is assumed that if a portion of arbitrary data is missing, then the regularization parameter ($\lambda_k$) should be able to predict the missing information. In this paper, we use the GCV function applied to the projected matrix $H_{k+1,k}$ from (\ref{eq:7}). The chosen regularization parameter using the GCV criterion minimizes the prediction error through the minimization of the GCV function:
\begin{equation} \label{eq:25}   
G(\lambda) = \frac{ k \| (I - H_{k+1,k}H_{k+1,k, \lambda}^{\dagger})\beta e_{1}\|^{2}} {trace(I - H_{k+1,k}H_{k+1,k, \lambda}^{\dagger})^2},
\end{equation}
where $H_{k+1,k, \lambda}^{\dagger} = (H_{k+1,k}^{T} H_{k+1,k} +\lambda^2 I)^{-1} H_{k+1,k}^{T}.$
Using the SVD $\displaystyle  H_{k+1,k} = U_k \Sigma_k V_k^T $, (\ref{eq:25}) can be rewritten as:
\begin{equation} \label{eq:26}
 G(\lambda) = \displaystyle \frac{ k \beta^2 \left(\displaystyle \sum_{i=1}^{k} \left(\frac{\lambda^2}{\sigma_i^2 + \lambda^2}[U^T_k e_1]_i\right)^2 + \left(\left[U^T_k e_1\right]_{k+1}\right)^2\right)}{\left(1+ \displaystyle \sum_{i=1}^k \frac{\lambda^2}{\sigma_i^2 + \lambda^2}\right)^2},
\end{equation}
and 
$\displaystyle
    \lambda_k = \arg \min_{\lambda}  G(\lambda)\,.
$

\subsubsection{Stopping Criterion}
We must determine a suitable criterion to stop the iterations of H-CMRH. Similar to \cite{chung2008weighted}, we would like to find a stopping iteration $k$ that minimizes:
\begin{equation} \label{eq:27} 
\displaystyle \frac{n \| (I - AA_{k}^{\dagger})b\|^2}{(trace((I-AA_k^{\dagger}))^2},
\end{equation}

where $A_{k}^{\dagger}$ is defined considering the approximate solution produced by hybrid CMRH. \rev{Assuming $x_0=0$, without loss of generality, this can be written as: } 
$$x_k = L_k y_{\lambda,k} = L_k(H^T_{k+1,k}H_{k+1,k} + \lambda ^2 I)^{-1} H_{k+1,k}^T L_{k+1}^{\dagger}b \equiv A^{\dagger}_kb.$$ 
Since $L_k$ does not have orthonormal columns, it is not straightforward to compute (\ref{eq:27}) exactly. However, it can be approximated using a formal parallelism with the correspondent derivations for other Krylov methods, see, e.g. \cite{chung2008weighted}. This involves
\begin{eqnarray*}
    n \|(I - AA^{\dagger}_k) b\|^2 &\approx& n \|L_{k+1}^{\dagger}(I - AA^{\dagger}_k) b\|^2\\
    &=& n \|(I-H_{k+1,k}(H^T_{k+1,k}H_{k+1,k} + \lambda_{k}^2 I)^{-1}H_{k+1,k}^T)L_{k+1}^{\dagger} b\|^2.
\end{eqnarray*}
where $L_{k+1}^{\dagger}AL_{k}=L_{k+1}^{\dagger}L_{k+1}H_{k+1,k}=H_{k+1,k}$ and $L_{k+1}^{\dagger} I b = \beta e_1$.
Using the SVD of $H_{k+1,k}$, the previous expression can be rewritten as:  
\begin{equation}
n \|{L_{k+1}^{\dagger}}(I - AA^{\dagger}_k) b\|^2 = 
n \beta ^2 \displaystyle\left(\sum_{i=1}^k \left(\frac{\lambda_k^2}{\sigma_i^2 + \lambda_k^2}[U_k^T e_1]_i\right)^2 + \left([U_k^Te_1]_{k+1}\right)^2\right). 
\end{equation}
The denominator of (\ref{eq:27}) can be similarly approximated by: 
\begin{equation}
\left(trace(I - AA_k^{\dagger})\right)^2 \approx \left((n-k) + \displaystyle \sum_{i=1}^k \frac{\lambda_k^2}{\sigma_i ^2 + \lambda_k ^2}\right)^2.
\end{equation}
Thus, (\ref{eq:27}) can be approximated by the following:
\begin{equation} 
\label{eq:28} 
\hat{G}(k) = \displaystyle \frac{ n \beta ^2 \displaystyle\left(\sum_{i=1}^k \left(\frac{\lambda_k^2}{\sigma_i^2 + \lambda_k^2}[U_k^T e_1]_i\right)^2 + \left([U_k^Te_1]_{k+1}\right)^2\right)}{ \left((n-k) + \displaystyle \sum_{i=1}^k \frac{\lambda_k^2}{\sigma_i ^2 + \lambda_k ^2}\right)^2}\,. 
\end{equation} 
$\hat{G}(k)$ is used to determine the stopping iteration. The 
 algorithm is terminated when the difference between values is very small: 
\begin{equation} \label{eq:tol}
\displaystyle{
\left|
\frac{\hat{G}(k+1) - \hat{G}(k)}
{\hat{G}(1)}
\right| < tol,
}
\end{equation}
\rev{or when the minimum of $\hat{G}(k)$ continues to increase within a certain window of iterations.}

\section{Numerical Results}\label{sec:numerics} \rev{Three numerical experiments are presented in this section to demonstrate the effectiveness of CMRH and H-CMRH in the context of ill-posed problems. First, we show a comparison between CMRH and other inner product free methods for system matrices $A$ with different properties. Next, we showcase its potential benefits with respect to other traditional methods in the simulated context of low precision arithmetic. Finally, we add a comparison between H-CMRH and hybrid GMRES for different image deblurring examples with  multiple noise levels, and a seismic tomography simulated example to illustrate its applicability to large-scale problems.
In all examples, white Gaussian noise was added to the measurements and we define the corresponding noise level as
$$
\texttt{nl}= \frac{\|e\|}{\|Ax\true\|}.
$$
Computations for all numerical experiments and illustrations described in this paper were done using MATLAB on an M-series MacBook Pro.}
\rev{\subsection{Comparison with other inner product free methods} 
In this section, we present a comparison between CMRH and other inner product free regularization methods:
Landweber, 
Richardson (first-order) 
and Chebychev 
(suitable for any non-singular symmetric or non-symmetric linear system). 
Note that in the case of very ill-posed problems, the latter is numerically equivalent to the second-order Richardson method \cite[7.2.5.]{Bjorck1996Numerical}, which uses a fixed step length. This is due to the fact that the fraction between the summation and the subtraction of the upper and lower bounds of the  eigenvalues is very close to 1 for ill-conditioned problems containing the system matrix of the normal equations.}

\rev{In the following experiments, we present the relative error norm histories that correspond to three examples with different spectral properties. First, we reintroduce test problem Spectra, where $A \in \R^{64 \times 64}$ is an SPD matrix and the smallest eigenvalues are numerically zero, making it very ill-conditioned. Next, we present an example where the system matrix is diagonally dominant, ill-conditioned, and tridiagonal. This is known as a Dorr matrix and it is from the MATLAB gallery. Here, $A \in \R^{256 \times 256}$ is a non symmetric matrix with real positive eigenvalues. The chosen solution corresponds to the previously introduced Heat example from the Regularization Tools Package~\cite{reg}.  Finally, we examine a symmetric matrix with negative eigenvalues. This matrix, where  $A \in \R^{256 \times 256}$, violates the theoretical convergence conditions of Landweber, Richardson, and Chebyshev. The example corresponds to Deriv2 and it is from the Regularization Tools Package. In all scenarios, white Gaussian noise of appropriate level (e.g. $\texttt{nl}= 0.5\%$, $ \texttt{nl}=50\%$ and $\texttt{nl}=0.1\%$) has been added to the measurements to achieve challenging but resolvable problems. For all comparison methods, bounds on either the eigenvalues of $A$ (for symmetric matrices) or $A^T A$ are required.}

\rev{Assume $s_1$ and $s_n$ are the smallest and largest eigenvalues of $A$. We select the bounds for the Chebychev method to be $\min(s_1^2, s_n^2)$ and  $1.001\min(s_1^2, s_n^2)$, which follows the implementation in \cite[7.2.5.]{Bjorck1996Numerical}. Similar to \cite[7.7.2.]{Bjorck1996Numerical}, the step length for Landweber is $0.99/\max(s_1^2, s_n^2)$. Lastly, we select the step length for Richardson to be $0.99/(s_1+s_2)$, which follows the implementation in \cite[7.2.3.]{Bjorck1996Numerical}}

\begin{figure}[ht]
\begin{tabular}{ccc}
    {\scriptsize Spectra} &  {\scriptsize Dorr} &  {\scriptsize Deriv2}\\ 
   \includegraphics[width=3.85cm]{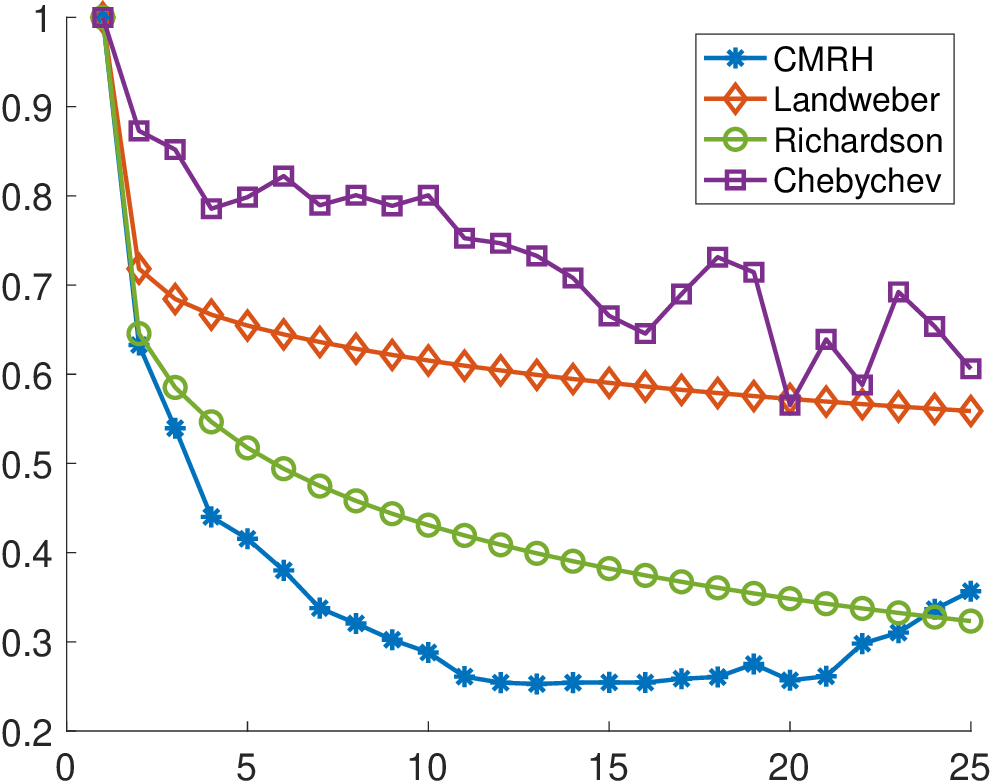} &  \includegraphics[width=3.85cm]{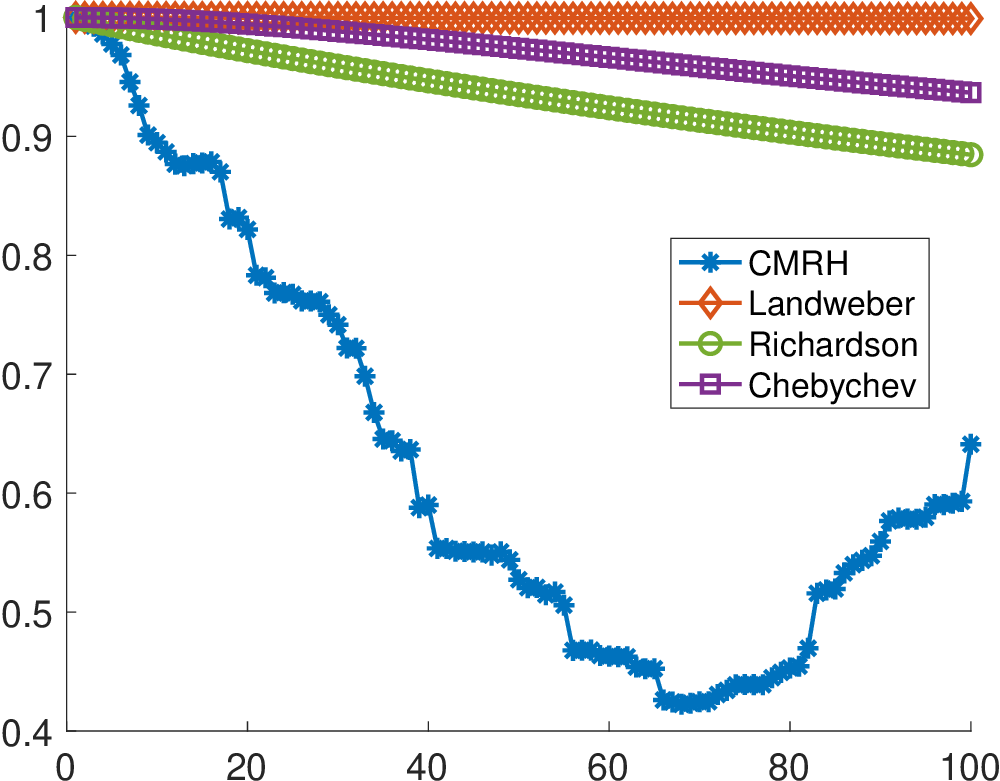} &
    \includegraphics[width=3.85cm]{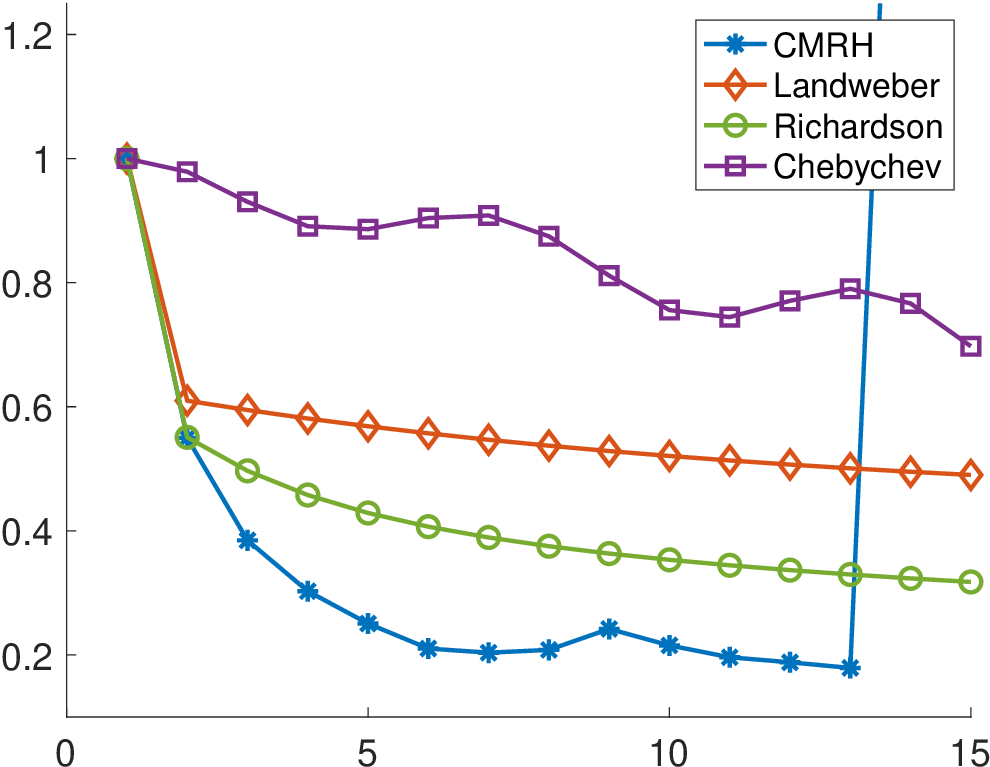}
\end{tabular}
\rev{\caption{Relative error norm histories for problems with different spectral properties.}}
\label{fig:no_inner}
\end{figure}

\rev{In Figure \ref{fig:no_inner}, it can be easily observed that CMRH is much faster than the other inner product free methods (particularly in the case of ill-posed problems). Moreover, CMRH does not require previous knowledge about the system matrix spectrum despite the price of increasing the memory requirements.}

\rev{\subsection{Results on low precision arithmetic} 
In this section we show a few instances of CMRH overcoming pathological behaviours arising when using GMRES in low precision arithmetic (low precision computations are simulated in MATLAB using Higham and Pranesh's {\tt chop} function \cite{higham2019simulating}). First, it should be noted that for large-scale inverse problems, such as those that arise in three-dimensional imaging applications, measured data contains noise, so one cannot expect to compute solutions to double-precision accuracy. Single-precision (32-bit) arithmetic generally provides sufficient accuracy and dynamic range; see, e.g., \cite{anand2009robust,li2024double,maas2021ct}.
However, some challenges might arise when the precision is reduced further; in the following, we highlight such problems using two examples. 
For example, inner products (or norms) can fail in low precision due to under- or overflow when the length of the vectors increases. This can be observed in Figure \ref{fig:low2}, where relative error norm histories are presented when using two NVIDIA quarter precision formats \cite{higham2019simulating}: test problem Deriv2 with $n=4096$ using q52, which has 5 exponent and 2 significand bits (here we observe underflow while computing the norm of $b$), and test problem Shaw with $n=6144$ using q43, which has 4 exponent and 3 significand bits (here we observe overflow while computing the norm of $b$). Test problems Deriv2 and Shaw are part of the Regularization Tools package~\cite{reg}. 
}

\begin{figure}[ht]
\centering
\begin{tabular}{cc}
    {\scriptsize Deriv2} &   {\scriptsize Shaw}\\ 
   \includegraphics[width=3.85cm]{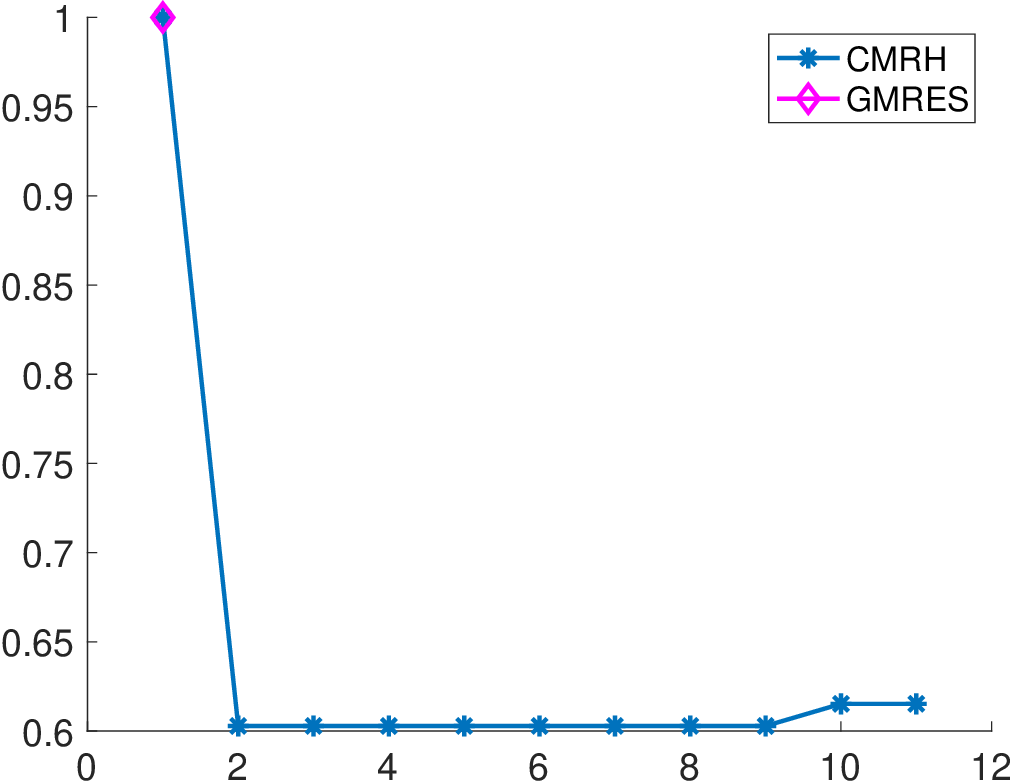} &  \includegraphics[width=3.85cm]{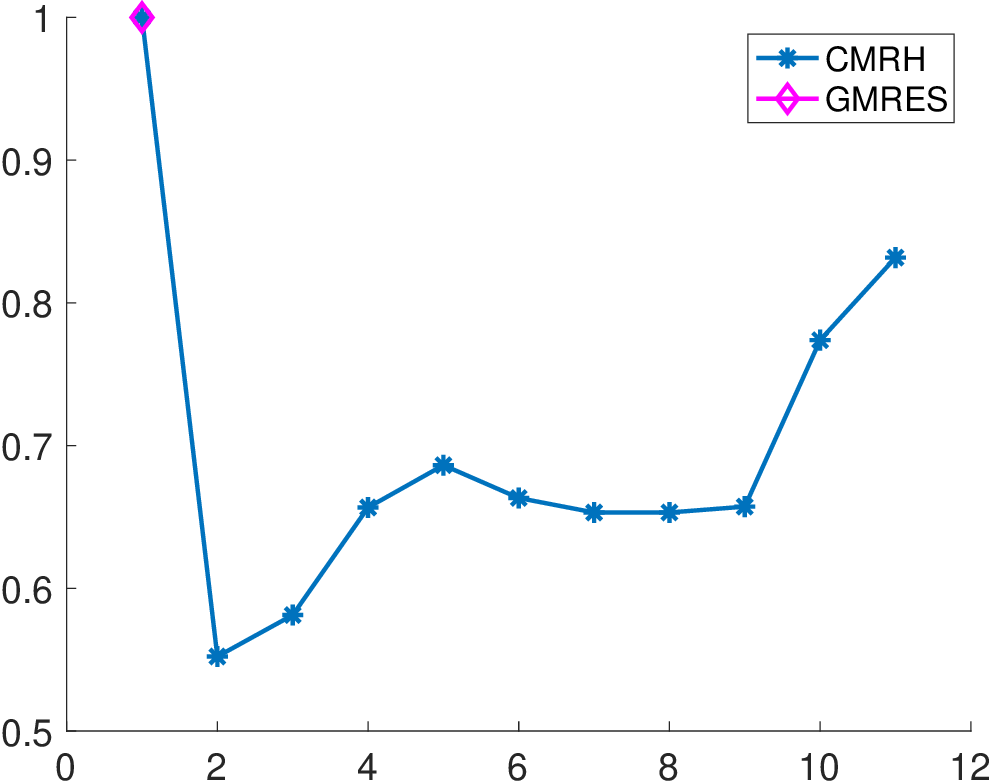} 
\end{tabular}
\rev{\caption{Relative error norm histories for Deriv2 and Shaw in precisions `q52' and `q43' (respectively). Early termination of GMRES is due to underflow (left) and overflow (right) computing vector norms.}}
\label{fig:low2}
\end{figure}

\rev{
Another major issue that can happen when using GMRES in low precision is that the basis vectors for the Krylov subspace can become orthogonal in the given precision causing the algorithm to stop too early. 
This is illustrated in Figure~\ref{fig:low1}, where we apply CMRH and GMRES to the Deriv2 test problem in half-precision arithmetic. In the left panel of Figure \ref{fig:low1} one can observe the relative error norm histories for CMRH and GMRES, where GMRES stops at iteration 4. An easy explanation of this phenomenon can be observed in the right panel of Figure \ref{fig:low1}, where the first 20 diagonal elements of the Hessenberg matrix $H^{A}$ produced by GMRES (computed in double precision) are presented along with the machine precision (in half precision) in logarithmic scale. One can easily observe that, regardless of the particular implementation, the elements in the diagonal of $H^{A}$ quickly fall below machine precision, 
which prompts an early termination of the algorithm as previously observed.}

\begin{figure}[ht]
\centering
\begin{tabular}{cc}
   \includegraphics[width=3.85cm]{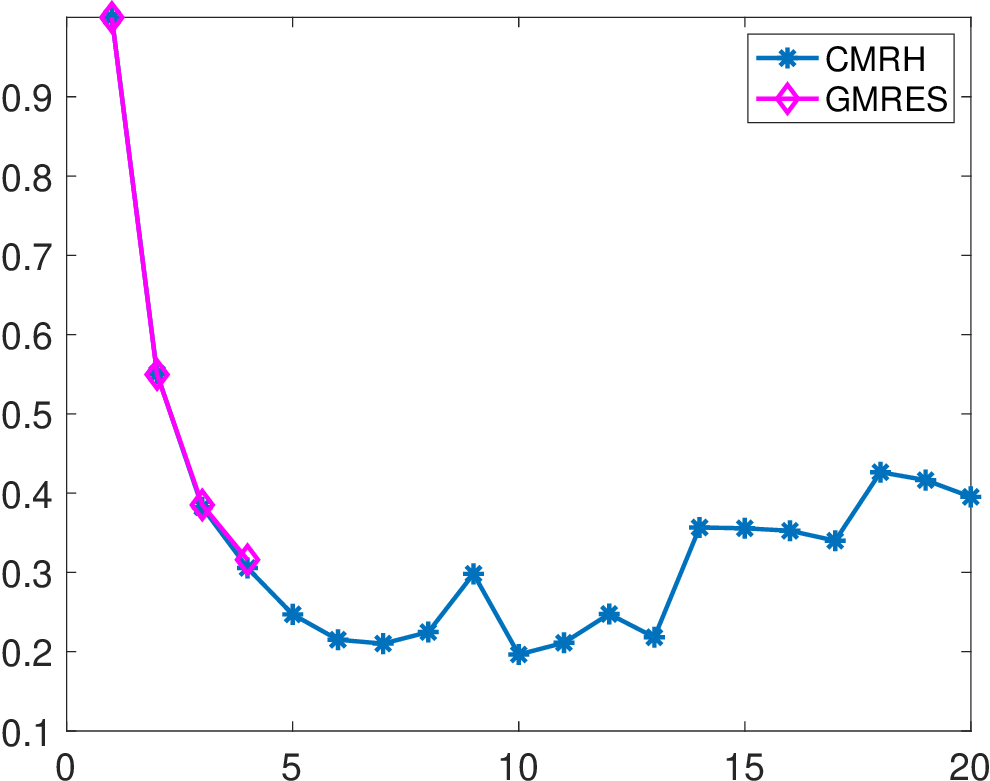} &  \includegraphics[width=3.85cm]{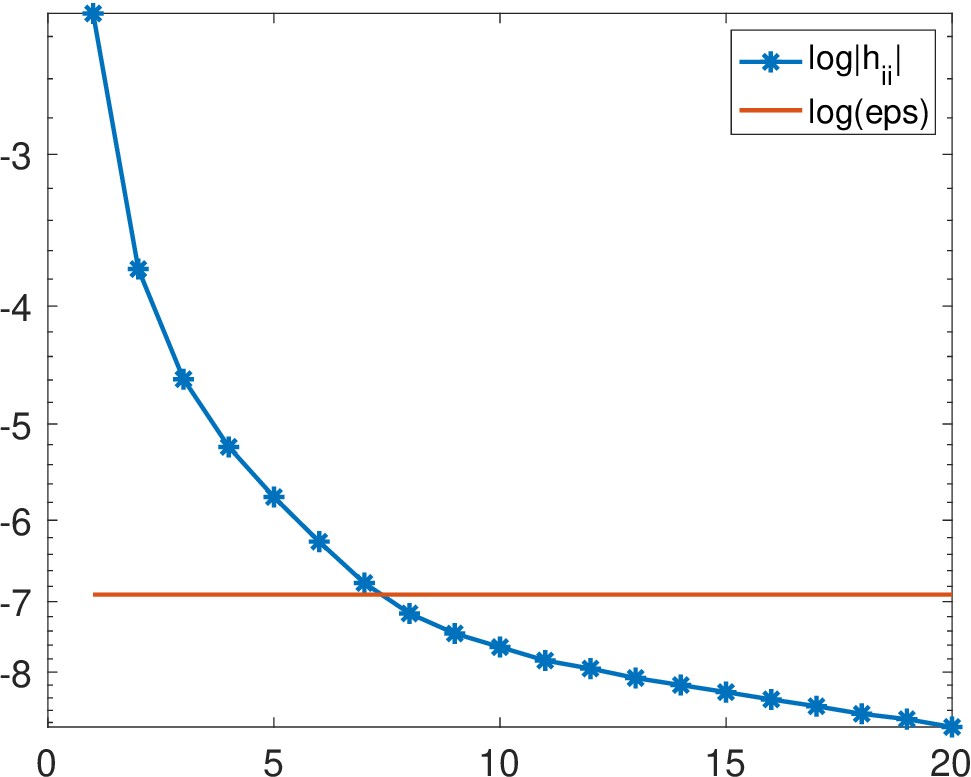} 
\end{tabular}
\rev{\caption{Test problem Deriv2. Left panel: relative error norm histories for GMRES and CMRH, where early stopping of GMRES is caused by numerical orthogonality of the basis vectors. Right panel: first 20 diagonal elements of the Hessenberg (Arnoldi) matrix produced by GMRES (computed in double precision) and machine (half) precision, in logarithmic scale.
}
\label{fig:low1}}
\end{figure}

\subsection{Results on deblurring test problems}\label{subsec:deblur}
To demonstrate the performance of H-CMRH, we consider three test problems: PRblur, PRblurshake, and PRblurspeckle. In particular, PRblur represents Gaussian noise, PRblurshake simulates random motion blur (shaking), and PRblurspeckle models atmospheric blur. 
These problems are 2D deblurring problems from the IRtools package \cite{gazzola2018ir} involving images with $256 \times 256$ pixels, corresponding to $65536 \times 65536$ system matrices $A$, and with added white Gaussian noise with noise levels between $\texttt{nl}=10^{-3}$ and $\texttt{nl}=10^{-1}.$

First, we show the blurred \rev{noisy} images (with $\texttt{nl}=10^{-2}$) along with the reconstructed images that were produced by H-CMRH with the GCV scheme;
\rev{see Figure \ref{fig:PR}. The reconstructions from hybrid GMRES look very similar, as might be expected from the relative errors shown in Figure~\ref{fig:PR2}, so we do not include them here. 
}

\begin{figure}[ht]
\centering
\begin{tabular}{ccc}
    {\scriptsize PRblur} &  {\scriptsize PRblurshake} &  {\scriptsize PRblurspeckle}\\ 
    \includegraphics[width=3.85cm]{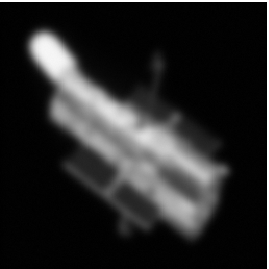} &  \includegraphics[width=3.85cm]{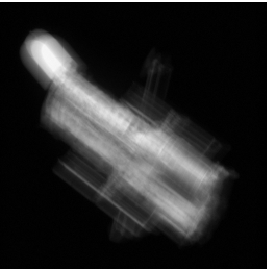} &
    \includegraphics[width=3.85cm]{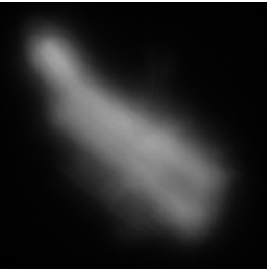} \\ 
\end{tabular} \\
\begin{tabular}{ccc}
    \includegraphics[width=3.85cm]{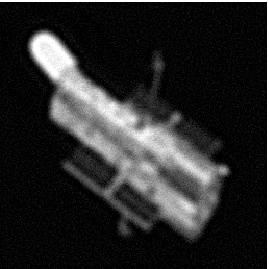} &  \includegraphics[width=3.85cm]{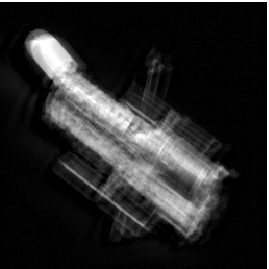} &
    \includegraphics[width=3.85cm]{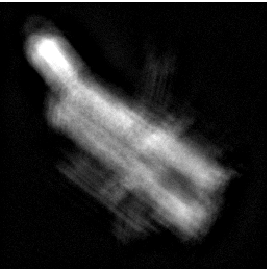} \\
\end{tabular}
\caption{Top: Blurred and noisy measurements with $\texttt{nl}=10^{-2}.$ Bottom: reconstructed images using H-CMRH.}
\label{fig:PR}
\end{figure}
Next, we compare the relative error norms for H-CMRH and hybrid GMRES with the GCV and optimal schemes for choosing the regularization parameter, which can be observed in Figure \ref{fig:PR2}. Note that the star marker represents where the H-CMRH algorithm stops based on the GCV stopping criteria in (\ref{eq:28}) and (\ref{eq:tol}).
\begin{figure}[ht]
\begin{tabular}{ccc}
    {\scriptsize PRblur} &  {\scriptsize PRblurshake} &  {\scriptsize PRblurspeckle}\\ 
   \includegraphics[width=3.92cm]{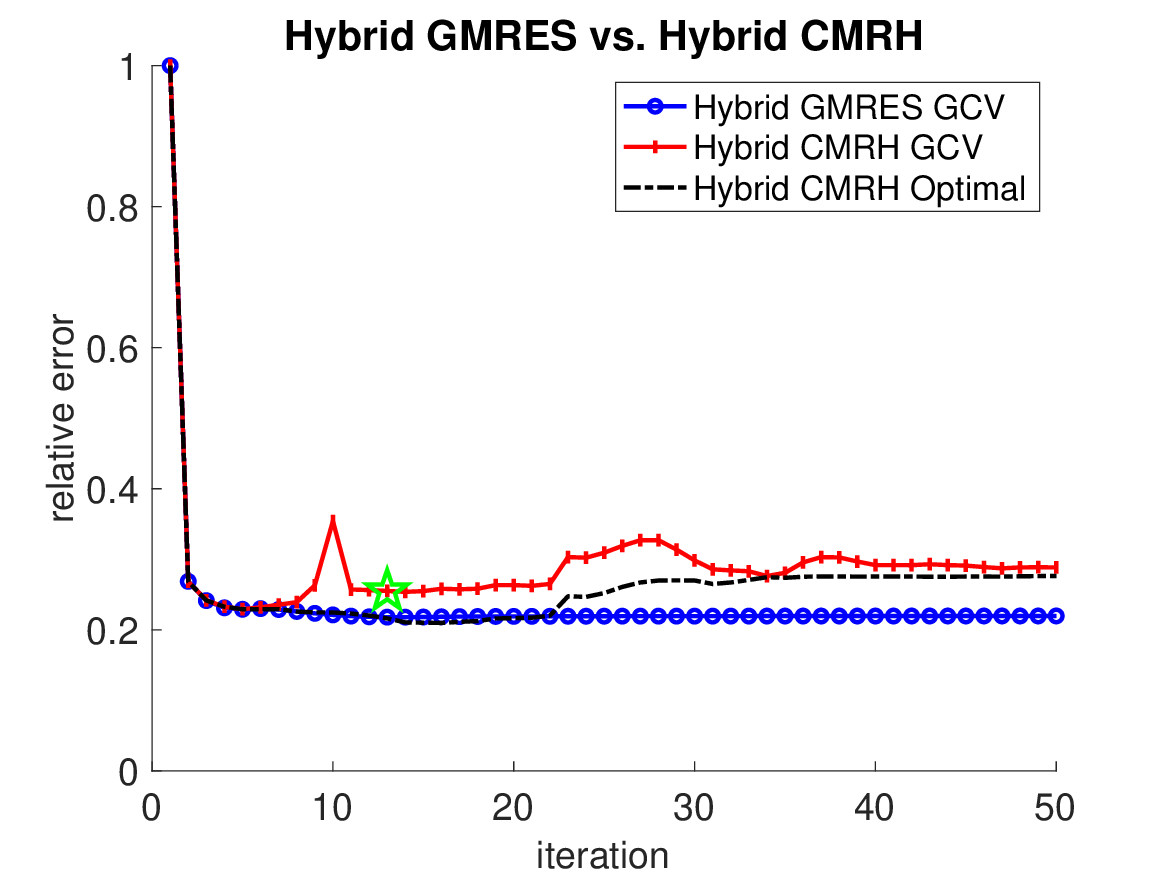} &  \includegraphics[width=3.92cm]{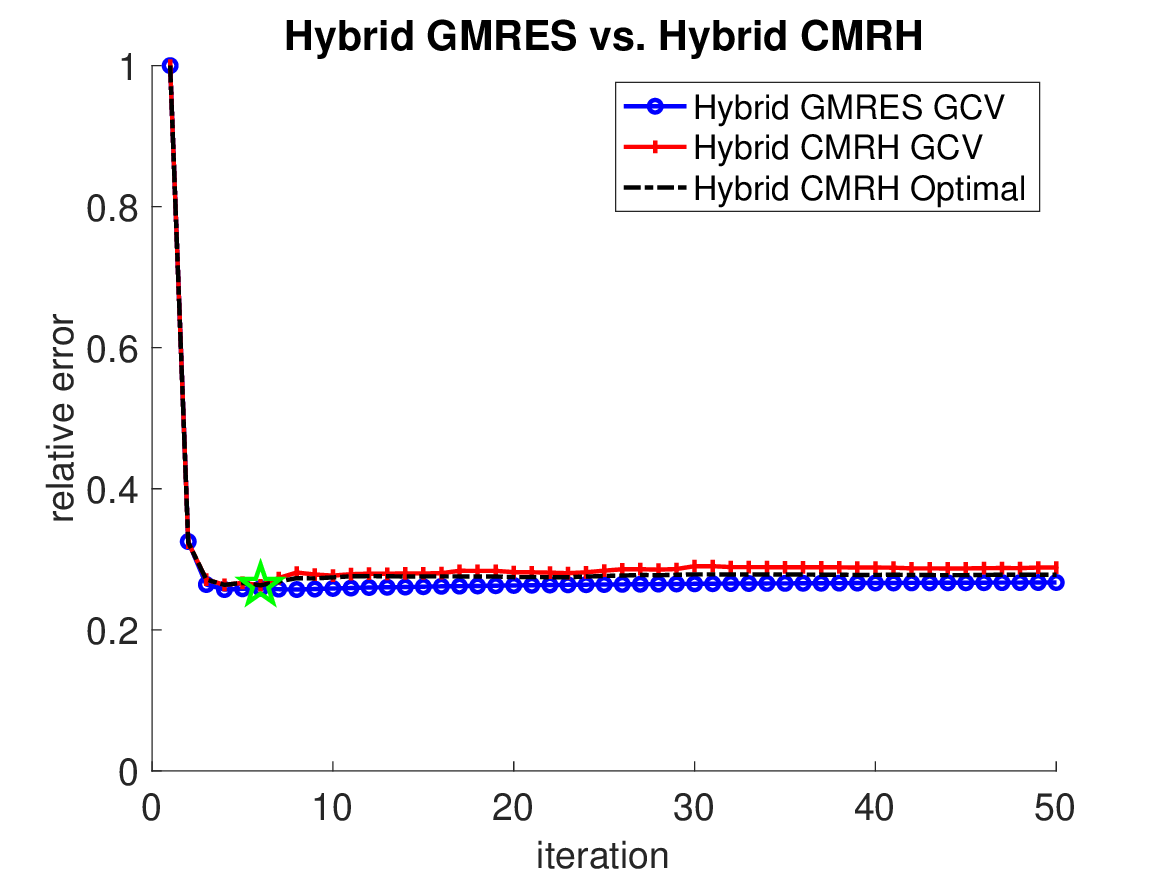} &
    \includegraphics[width=3.92cm]{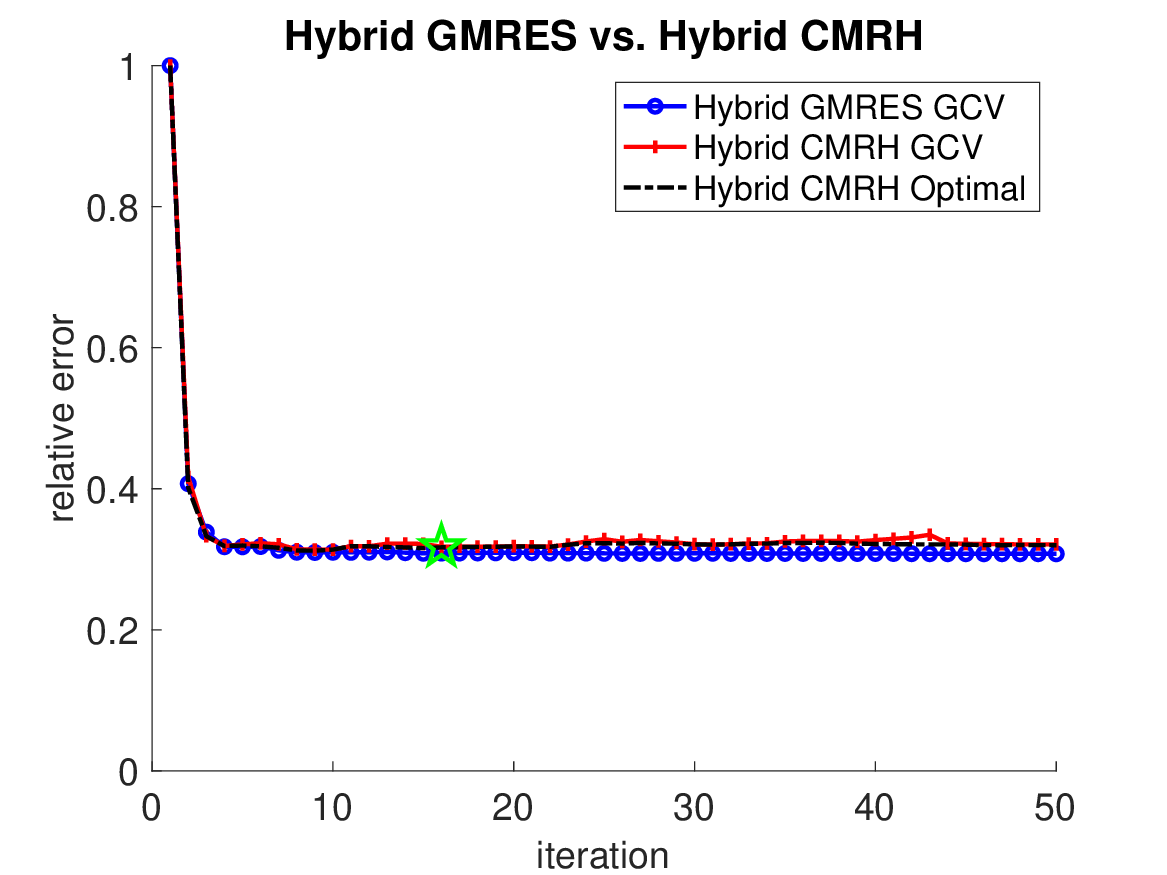}
\end{tabular}
\caption{Relative error norm for 2D deblurring problems using H-CMRH and hybrid GMRES.}
\label{fig:PR2}
\end{figure}
In this figure, we observe that H-CMRH and hybrid GMRES behave in almost an identical manner for PRblurshake and PRblurspeckle. In PRblur, this is also true for the iterations before the stopping iteration. 

\rev{Finally, additional experiments were performed with PRblur, PRblurshake, and PRblurspeckle using various noise levels. Table \ref{tab:Table1} displays the relative error norms for the solutions produced using the GCV stopping criteria for H-CMRH and hybrid GMRES. As proven in Section~\ref{sec:H-CMRH}, the test results shown in Figures \ref{fig:PR}, \ref{fig:PR2}, and Table \ref{tab:Table1} illustrate that H-CMRH and hybrid GMRES have similar regularizing properties, while H-CMRH is an inherent inner product free method. }

\begin{center}
\begin{table}[h!]
\begin{tabular}{ |c|c|c|c|c| } 
\hline
\multicolumn{5}{|c|}{PRblur} \\
\hline
Method & Noise Level & Stopping Iteration & Reg Parameter & Relative Error\\
\hline
\multirow{3}{4em}{Hybrid CMRH} & $10^{-3}$ & $8$ & $5.5042 \times 10^{-5}$ & $0.2060$ \\ 
& $10^{-2}$ & $12$ & $0.0589$ & $0.2550$ \\ 
& $10^{-1}$ & $6$ & $0.1396$ & $0.3098$ \\ 
\hline
\multirow{3}{4em}{Hybrid GMRES} & $10^{-3}$ & $14$ & $0.0124$ & $0.2016$ \\ 
& $10^{-2}$ & $12$ & $0.0561$ & $0.2179$ \\ 
& $10^{-1}$ & $5$ & $0.1930$ & $0.2493$ \\ 
\hline
\end{tabular} 
\\
\begin{tabular}{ |c|c|c|c|c| } 
\hline
\multicolumn{5}{|c|}{PRblurshake} \\
\hline
Method & Noise Level & Stopping Iteration & Reg Parameter & Relative Error\\
\hline
\multirow{3}{4em}{Hybrid CMRH} & $10^{-3}$ & $13$ & $0.0051$ & $0.2891$ \\ 
& $10^{-2}$ & $5$ & $3.7149 \times 10^{-5}$ & $0.2631$ \\ 
& $10^{-1}$ & $5$ & $3.7855 \times 10^{-5}$ & $0.3523$ \\ 
\hline
\multirow{3}{4em}{Hybrid GMRES} & $10^{-3}$ & $4$ & $0.1608$ & $0.2565$ \\ 
& $10^{-2}$ & $4$ & $0.1631$ & $0.2581$ \\ 
& $10^{-1}$ & $3$ & $0.2384$ & $0.3310$ \\ 
\hline
\end{tabular}

\begin{tabular}{ |c|c|c|c|c| } 
\hline
\multicolumn{5}{|c|}{PRblurspeckle} \\
\hline
Method & Noise Level & Stopping Iteration & Reg Parameter & Relative Error\\
\hline
\multirow{3}{4em}{Hybrid CMRH} & $10^{-3}$ & $5$ & $4.0987 \times 10^{-5}$ & $0.3167$ \\ 
& $10^{-2}$ & $15$ & $0.0051$ & $0.3177$ \\ 
& $10^{-1}$ & $5$ & $3.5293 \times 10^{-5}$ & $0.7842$ \\ 
\hline
\multirow{3}{4em}{Hybrid GMRES} & $10^{-3}$ & $5$ & $0.0771$ & $0.3153$\\ 
& $10^{-2}$ & $5$ & $0.0827$ & $0.3181$ \\ 
& $10^{-1}$ & $9$ & $0.2073$ & $0.3414$ \\ 
\hline
\end{tabular}

\caption{\rev{Numerical results for the three test problems PRblur, PRblurshake, and PRblurspeckle for various noise levels.  Regularization parameters and relative error norms correspond to values at the stopping iteration.}}
\label{tab:Table1}
\end{table}
\end{center}

\rev{\subsection{Results on a tomography test problem}  In this section we present an additional large-scale problem from the AIR Tool II package \cite{hansen2018air}, which can also be accessed via IRtools as PRseismic. This is a simulated seismic travel-time tomography test problem. Since PRseismic produces a rectangular system, we apply H-CMRH and hybrid GMRES to the system matrix corresponding to the normal equations, so that $A \in \R^{65536 \times 65536}$ where the images have $256 \times 256 $ pixels.}

\rev{Similar to the image deblurring test problems in Section~\ref{subsec:deblur}, we first show the measured noisy data, the true solution, and the reconstructed images that were produced by H-CMRH and hybrid GMRES with the GCV scheme at the stopping iteration; see Figure~\ref{fig:PRseismic}. Again, the reconstructions obtained using hybrid GMRES and H-CMRH have a similar quality.
The relative error norm histories are shown in Figure \ref{fig:PRseismic_rel_error}. 
Although in this particular example, H-CMRH terminates (at iteration 9) slightly before the optimal stopping iteration, this could be adjusted by modifying tolerances used in the GCV criterion. The GCV stopping criteria for hybrid GMRES (stopping at iteration 18) performed slightly better in this example. 
}

\begin{figure}[ht]
\centering
\begin{tabular}{cc}
    {\large PRseismic} &  {\large PRseismic} \\ 
    \includegraphics[width=3.85cm]{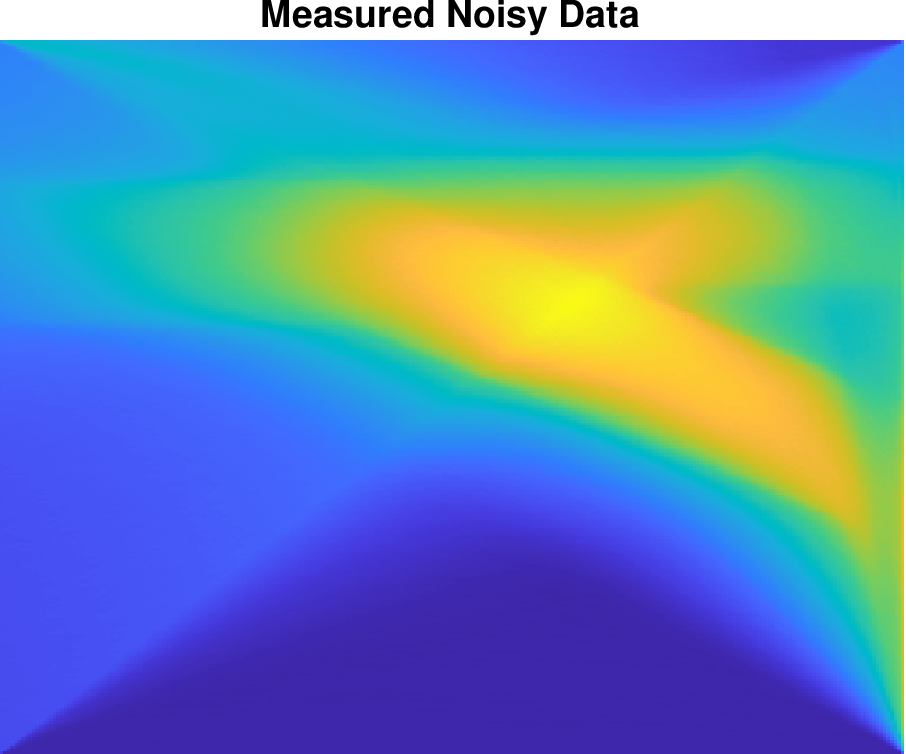} &  \includegraphics[width=3.85cm]{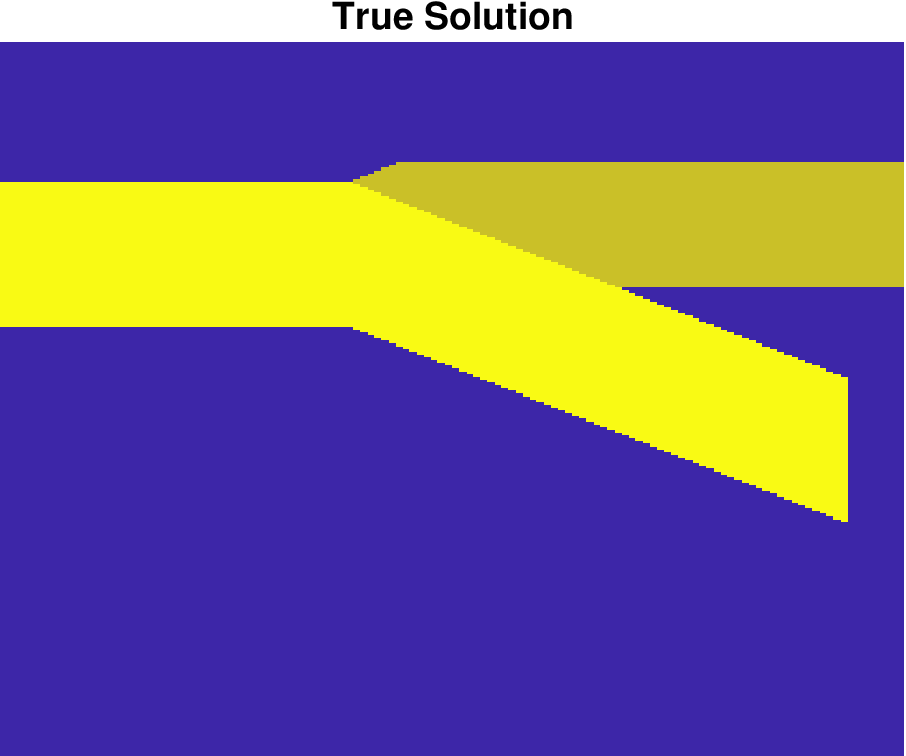}\\ 
\end{tabular} \\
\begin{tabular}{cc}
    \includegraphics[width=3.85cm]{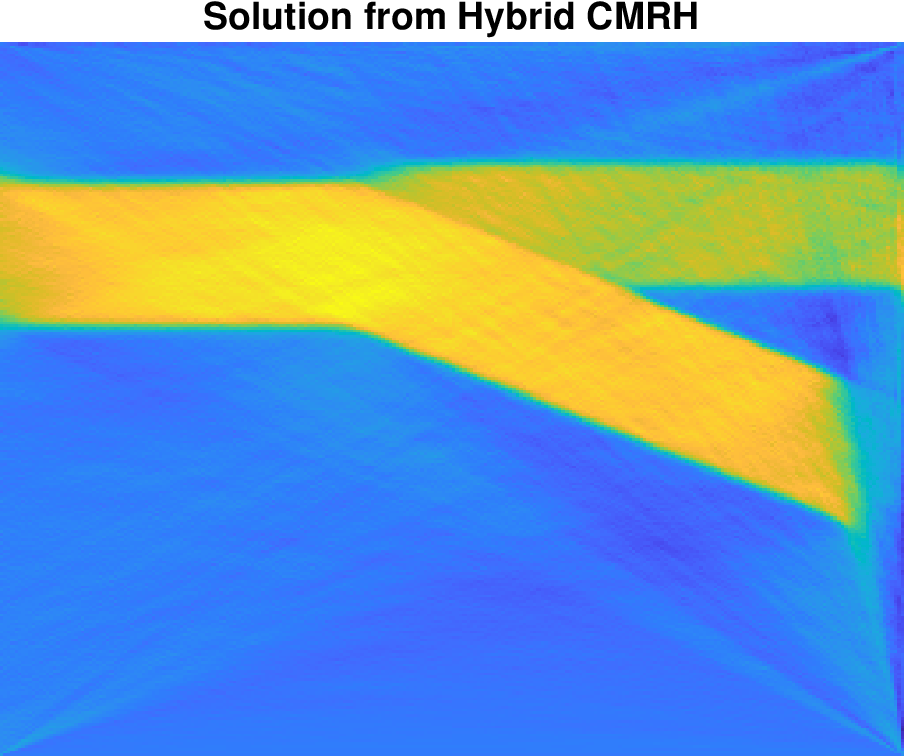} &  \includegraphics[width=3.85cm]{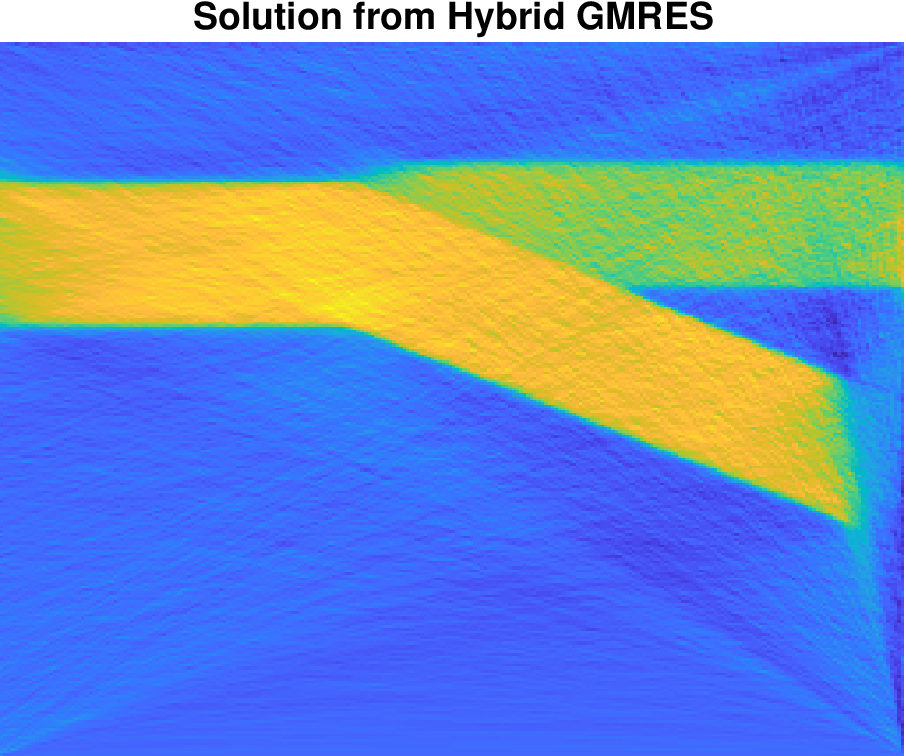}\\
\end{tabular}
\caption{Measured noisy data (top left), true solution (top right), and reconstructed images using H-CMRH and hybrid GMRES (bottom row).}
\label{fig:PRseismic}
\end{figure}

\begin{figure}[ht]
\centering
\begin{tabular}{c}
   \includegraphics[width=6.5cm]{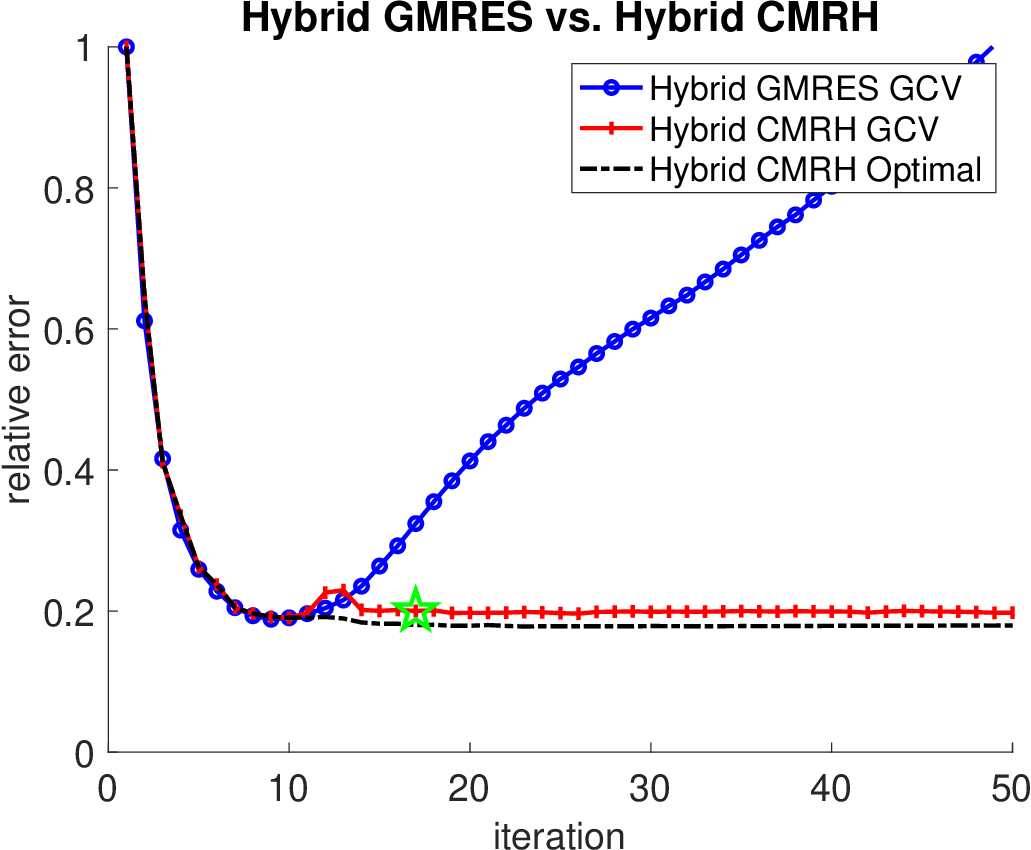} \\ 
\end{tabular} \\
\caption{Relative error norm for seismic problem with $\texttt{nl}=10^{-1}$ using H-CMRH and hybrid GMRES.}
\label{fig:PRseismic_rel_error}
\end{figure}

\rev{Additional experiments were conducted with different noise levels. Table \ref{tab:Table2} provides the relative error norm produced by H-CMRH and hybrid GMRES using the GCV stopping criteria. For $\texttt{nl}= 0.10\%$, H-CMRH greatly outperforms hybrid GMRES. Therefore, we find that H-CMRH performs well when the noise level is increased.}

\begin{center}
\begin{table}[h!]
\begin{tabular}{ |c|c|c|c|c| } 
\hline
\multicolumn{5}{|c|}{ Numerical Results for PRseismic} \\
\hline
Method & Noise Level & Stopping Iteration & Reg Parameter & Relative Error\\
\hline
\multirow{3}{4em}{Hybrid CMRH} & $10^{-3}$ & $39$ & $0.0051$ & $0.0947$ \\ 
& $10^{-2}$ & $9$ & $6.1839 \times 10^{-5}$ & $0.1645$ \\ 
& $10^{-1}$ & $16$ & $4.0069 \times 10^{3}$ & $0.2001$ \\ 
\hline
\multirow{3}{4em}{Hybrid GMRES} & $10^{-3}$ & $17$ & $220.394$ & $0.1175$ \\ 
& $10^{-2}$ & $18$ & $216.23$ & $0.1196$ \\ 
& $10^{-1}$ & $27$ & $121.54$ & $0.5822$ \\ 
\hline
\end{tabular}
\caption{Numerical results for PRseismic with various noise levels.  Regularization parameters and relative errors correspond to values at the stopping iteration.} \label{tab:Table2}
\end{table}
\end{center}

\section{Conclusion}\label{sec:conclusions}

\rev{In conclusion, this work establishes CMRH as a robust iterative regularization method, expanding its potential to a diverse range of applications in large-scale ill-posed problems. Through a detailed study, we offer both theoretical insights and a deeper understanding of the properties of the projected problems involved. The regularization characteristics of CMRH are shown to effectively filter solutions, a conclusion that we support using empirical evidence. 
Once this method is proven to be apt
for ill-posed problems, we show that CMRH is able to deliver a solution of similar quality much faster than other inner product free alternatives, and highlight the advantages of using inner product free methods in low-precision arithmetic scenarios, where CMRH overcomes certain limitations of GMRES. Moreover, we introduce a novel hybrid version of the CMRH method (H-CMRH), the first hybrid method to be inner product free. 
We again stress that it was first necessary to demonstrate the regularization capabilities of CMRH before proposing the hybrid verision, H-CMRH. 
Finally, the performance of CMRH and H-CMRH is validated through its application to various ill-posed problems.}

\rev{In addition, we believe that CMRH is well-suited to be used in combination with randomized methods for least squares problems. This is largely due to the common assumption in randomized linear algebra that orthogonalization poses a computational bottleneck which can be alleviated using randomization techniques. Although a detailed exploration of this idea falls outside the scope of this paper, we note its potential and leave this for future investigation.
}

\bibliographystyle{siamplain}
\bibliography{references}

\begin{thebibliography}{10}

\bibitem{anand2009robust}
{\sc C.~K. Anand}, {\em Robust solvers for inverse imaging problems using dense
  single-precision hardware}, J.~Math.~Imaging~Vis., 33 (2009), pp.~105--120.

\bibitem{1994templates}
{\sc R.~Barrett, M.~Berry, T.~F. Chan, J.~Demmel, J.~Donato, J.~Dongarra,
  V.~Eijkhout, R.~Pozo, C.~Romine, and H.~van~der Vorst}, {\em Templates for
  the Solution of Linear Systems: Building Blocks for Iterative Methods},
  Society for Industrial and Applied Mathematics, 1994,
  \url{https://doi.org/10.1137/1.9781611971538}.

\bibitem{Bauer2011Param}
{\sc F.~Bauer and M.~Lukas}, {\em Comparing parameter choice methods for
  regularization of ill-posed problem}, Mathematics and Computers in
  Simulation, 81 (2011), pp.~1795--1841,
  \url{https://doi.org/10.1016/j.matcom.2011.01.016}.

\bibitem{Bjorck1996Numerical}
{\sc A.~Björck}, {\em Numerical Methods for Least Squares Problems}, Society
  for Industrial and Applied Mathematics, 1996,
  \url{https://doi.org/10.1137/1.9781611971484}.

\bibitem{chung2024computational}
{\sc J.~Chung and S.~Gazzola}, {\em Computational methods for large-scale
  inverse problems: a survey on hybrid projection methods}, SIAM Review, 66
  (2024), pp.~205--284.

\bibitem{Chung2011Inverse}
{\sc J.~Chung, S.~Knepper, and J.~G. Nagy}, {\em Large-scale inverse problems
  in imaging}.
\newblock In Handbook of Mathematical Methods in Imaging, O. Scherzer (Ed.).
  Springer, New York, NY, 2015,
  \url{https://doi.org/10.1007/978-0-387-92920-0\_2}.

\bibitem{chung2008weighted}
{\sc J.~Chung, J.~G. Nagy, and D.~P. O’Leary}, {\em A weighted {GCV} method
  for {L}anczos hybrid regularization}, Electronic Transactions on Numerical
  Analysis, 28 (2008), pp.~149--167.

\bibitem{sadok2016algorithms}
{\sc S.~Duminil, M.~Heyounni, P.~Marion, and H.~Sadok}, {\em {A}lgorithms for
  the {CMRH} method for dense linear systems}, Numerical Algorithms, 71 (2016),
  pp.~383--394.

\bibitem{gazzola2018ir}
{\sc S.~Gazzola, P.~C. Hansen, and J.~G. Nagy}, {\em {IR} {T}ools: a {MATLAB}
  package of iterative regularization methods and large-scale test problems},
  Numerical Algorithms,  (2018), pp.~1--39.

\bibitem{Gazzola2015OnKP}
{\sc S.~Gazzola, P.~Novati, and M.~R. Russo}, {\em On {K}rylov projection
  methods and {T}ikhonov regularization}, Electronic Transactions on Numerical
  Analysis, 44 (2015), pp.~83--123.

\bibitem{Gazzola2020Krylov}
{\sc S.~Gazzola and M.~Sabaté~Landman}, {\em {K}rylov methods for inverse
  problems: Surveying classical, and introducing new, algorithmic approaches},
  GAMM-Mitteilungen, 43 (2020), p.~e202000017,
  \url{https://doi.org/https://doi.org/10.1002/gamm.202000017}.

\bibitem{VARGA1961}
{\sc G.~H. Golub and R.~S. Varga}, {\em {C}hebyshev semi-iterative methods,
  successive overrelaxation iterative methods, and second order {R}ichardson
  iterative methods. {P}arts {I} and {II}.}, Numerische Mathematik, 3 (1961),
  pp.~147--156.

\bibitem{reg}
{\sc P.~C. Hansen}, {\em Regularization tools: A {MATLAB} package for analysis
  and solution of discrete ill-posed problems}, Numerical algorithms, 6 (1994),
  pp.~1--35.

\bibitem{Hansen2010}
{\sc P.~C. Hansen}, {\em Discrete Inverse Problems: Insight and Algorithms},
  SIAM, Philadelphia, 2010.

\bibitem{hansen2018air}
{\sc P.~C. Hansen and {J{\o}rgensen}}, {\em {AIR} {T}ools {II}: algebraic
  iterative reconstruction methods, improved implementation}, Numerical
  Algorithms, 79 (2018), pp.~107--137.

\bibitem{Hansen2006deblurring}
{\sc P.~C. Hansen, J.~G. Nagy, and D.~P. O'Leary}, {\em Deblurring Images},
  Society for Industrial and Applied Mathematics, 2006,
  \url{https://doi.org/10.1137/1.9780898718874}.

\bibitem{sadok2008new}
{\sc M.~Heyouni and H.~Sadok}, {\em {A} new implementation of the {CMRH} method
  for solving dense linear systems}, Journal of Computational and Applied
  Mathematics, 213 (2008), pp.~387--399.

\bibitem{higham2019simulating}
{\sc N.~J. Higham and S.~Pranesh}, {\em Simulating low precision floating-point
  arithmetic}, SIAM Journal on Scientific Computing, 41 (2019), pp.~C585--C602.

\bibitem{Kilmer2001RegParam}
{\sc M.~Kilmer and D.~O'Leary}, {\em Choosing regularization parameters in
  iterative methods for ill-posed problems}, SIAM Journal on Matrix Analysis
  and Applications, 22 (2001), \url{https://doi.org/10.1137/S0895479899345960}.

\bibitem{li2024double}
{\sc H.~Li}, {\em Double precision is not necessary for {LSQR} for solving
  discrete linear ill-posed problems}, J.~Sci.~Comput., 98 (2024), p.~55.

\bibitem{maas2021ct}
{\sc C.~{Maa{\ss}} and {{\SS}aer, M. and Kachelrie{\ss}}}, {\em {CT} image
  reconstruction with half precision floating-point values}, Medical Physics,
  38 (2011), pp.~S95--S105.

\bibitem{Nagy2003SD}
{\sc J.~G. Nagy and K.~M. Palmer}, {\em Steepest descent, {CG}, and iterative
  regularization of ill-posed problems}, BIT, 43 (2003), p.~1003–1017,
  \url{https://doi.org/10.1023/B:BITN.0000014546.51341.53}.

\bibitem{Reichel2013Param}
{\sc L.~Reichel and G.~Rodriguez}, {\em Old and new parameter choice rules for
  discrete ill-posed problems}, Numerical Algorithms, 63 (2013),
  \url{https://doi.org/10.1007/s11075-012-9612-8}.

\bibitem{Saad1989Supercomputers}
{\sc Y.~Saad}, {\em Krylov subspace methods on supercomputers}, SIAM Journal on
  Scientific and Statistical Computing, 10 (1989), pp.~1200--1232,
  \url{https://doi.org/10.1137/0910073}, \url{https://doi.org/10.1137/0910073},
  \url{https://arxiv.org/abs/https://doi.org/10.1137/0910073}.

\bibitem{saad2003iterative}
{\sc Y.~Saad}, {\em Iterative Methods for Sparse Linear Systems}, Society for
  Industrial and Applied Mathematics, Philadelphia, PA, second~ed., 2003,
  \url{https://doi.org/10.1137/1.9780898718003}.

\bibitem{sadok1999new}
{\sc H.~Sadok}, {\em {CMRH}: {A} new method for solving nonsymmetric linear
  systems based on the {H}essenberg reduction algorithm}, Numerical Algorithms,
  20 (1999), pp.~303--321.

\bibitem{sadok2011new}
{\sc H.~Sadok and D.~B. Szyld}, {\em {A} new look at {CMRH} and its relation to
  {GRMES}}, BIT Numerical Mathematics, 52 (2011), pp.~485--501.

\bibitem{Trussell1983Convergence}
{\sc H.~J. Trussell}, {\em Convergence criteria for iterative restoration
  methods}, {IEEE} Transactions on Acoustics, Speech, and Signal Processing,
  ASSP-31 (1983), pp.~129--136.

\bibitem{Vogel2002}
{\sc C.~R. Vogel}, {\em Computational Methods for Inverse Problems}, SIAM,
  Philadelphia, 2002.

\bibitem{Zhdanov2002}
{\sc M.~Zhdanov}, {\em Geophysical Inverse Theory and Regularization Problem},
  vol.~36, Elsevier, 01 2002.

\end{thebibliography}
\end{document}